 \documentclass[final,3p,times]{elsarticle}
\usepackage{amsfonts}
\usepackage{ulem}
\usepackage{cancel}
\usepackage{graphicx}
\usepackage{exscale}
\usepackage{times}
\usepackage{color}
\usepackage{amssymb}
\usepackage{setspace}
\usepackage{natbib}
\usepackage{amsmath}
\usepackage{bm}
\usepackage[bbgreekl]{mathbbol}
\usepackage{breqn}

\font\pppppcarac=ptmr8y at 5pt
\font\ppppcarac=ptmr8y at 6pt

\font\pcarac=ptmr8y at 9pt

\font\Ppcarac=ptmr8y at 10pt

\font\bf=ptmb8y at 10pt

\newcommand{\bfC}{{\bm{C}}}

\newcommand{\bfL}{{\bm{L}}}

\newcommand{\bfS}{{\bm{S}}}

\newcommand{\bfU}{{\bm{U}}}

\newcommand{\bfW}{{\bm{W}}}
\newcommand{\bfX}{{\bm{X}}}
\newcommand{\bfY}{{\bm{Y}}}
\newcommand{\bfZ}{{\bm{Z}}}

\newcommand{\bfzero}{{ \hbox{\bf 0} }}

\newcommand{\bfk}{{\bm{k}}}

\newcommand{\bfs}{{\bm{s}}}

\newcommand{\bfu}{{\bm{u}}}
\newcommand{\bfv}{{\bm{v}}}
\newcommand{\bfw}{{\bm{w}}}
\newcommand{\bfx}{{\bm{x}}}
\newcommand{\bfy}{{\bm{y}}}
\newcommand{\bfz}{{\bm{z}}}

\newcommand{\bfLambda}{{\bm{\Lambda}}}

\newcommand{\bfPhi}{{\bm{\Phi}}}

\newcommand{\bfXi}{{\bm{\Xi}}}

\newcommand{\bfbeta}{{\bm{\beta}}}

\newcommand{\bflambda}{{\bm{\lambda}}}

\newcommand{\bfomega}{{\bm{\omega}}}
\newcommand{\bfvarphi}{{\bm{\varphi}}}

\newcommand{\bftau}{{\bm{\tau}}}
\newcommand{\bfxi}{{\bm{\xi}}}
\newcommand{\bfzeta}{{\bm{\zeta}}}

\newcommand{\CC}{{\mathbb{C}}}

\newcommand{\HH}{{\mathbb{H}}}

\newcommand{\KK}{{\mathbb{K}}}
\newcommand{\LL}{{\mathbb{L}}}
\newcommand{\MM}{{\mathbb{M}}}
\newcommand{\NN}{{\mathbb{N}}}

\newcommand{\RR}{{\mathbb{R}}}

\newcommand{\cc}{{\mathbb{c}}}

\newcommand{\GammaGamma}{{\mathbb{\Gamma}}}

\newcommand{\gammagamma}{{\mathbb{\gamma}}}

\DeclareMathAlphabet{\mathonebb}{U}{bbold}{m}{n}
\def\11{{\ensuremath{\mathonebb{1}}}}

\newcommand{\curB}{{\mathcal{B}}}
\newcommand{\curC}{{\mathcal{C}}}

\newcommand{\curG}{{\mathcal{G}}}
\newcommand{\curL}{{\mathcal{L}}}

\newcommand{\curN}{{\mathcal{N}}}
\newcommand{\curP}{{\mathcal{P}}}
\newcommand{\curS}{{\mathcal{S}}}
\newcommand{\curT}{{\mathcal{T}}}

\newcommand{\bfcurS}{ {\boldsymbol{{\mathcal{S}}}} }

\newcommand{\peff}{{\hbox{{\ppppcarac eff}}}}

\newcommand{\sgn}{\hbox{{\Ppcarac sgn}}\,}

\newcommand{\pmin}{{\hbox{{\ppppcarac min}}}}
\newcommand{\pmax}{{\hbox{{\ppppcarac max}}}}

\newcommand{\tr}{{\hbox{{\textrm tr}}}}

\newcommand{\psim}{{\hbox{{\ppppcarac sim}}}}
\newcommand{\ppsim}{{\hbox{{\pppppcarac sim}}}}

\newcommand{\pconv}{\hbox{{\pcarac conv}}\,}
\newcommand{\unc}{{\hbox{{\ppppcarac unc}}}}

\newdefinition{definition}{Definition}
\newtheorem{lemma}{Lemma}
\newtheorem{proposition}{Proposition}
\newtheorem{corollary}{Corollary}
\newproof{proof}{Proof}
\newdefinition{remark}{Remark}
\newdefinition{hypothesis}{Hypothesis}
\newdefinition{notation}{Notation}
\newproof{example}{Example}
\numberwithin{equation}{section}

\journal{arXiv}

\begin{document}

\begin{frontmatter}
\title{Stochastic elliptic operators defined by non-gaussian random fields\\
 with uncertain spectrum}

\author[1]{C. Soize \corref{cor1}}
\ead{christian.soize@univ-eiffel.fr}

\cortext[cor1]{Corresponding author: C. Soize, christian.soize@univ-eiffel.fr}
\address[1]{Universit\'e Gustave Eiffel, MSME UMR 8208 CNRS, 5 bd Descartes, 77454 Marne-la-Vall\'ee, France}

\begin{abstract}
This paper present a construction and the analysis of a class of non-Gaussian positive-definite matrix-valued homogeneous random fields
with uncertain spectral measure for stochastic elliptic operators. Then the stochastic elliptic boundary value problem in a bounded domain of the 3D-space is introduced and analyzed for stochastic homogenization.
\end{abstract}


\end{frontmatter}

\section{Introduction}
\label{sec:Section1}
Random fields theory has extensively been developed \cite{Yadrenko1983,Rozanov1998,Adler2010,Vanmarcke2010,Rosenblatt2012}, in particular in the context of continuum physics \cite{Ostoja1998,Soize2017b,Malyarenko2018}.
The framework of this paper is that of the analysis of the stochastic homogenization of a 3D-linear anisotropic elastic random medium. The elasticity field is modeled by a Non-Gaussian positive-definite fourth-order tensor-valued homogeneous random field. This paper present an extension of the works \cite{Soize2006,Nouy2014,Soize2017e,Soize2017b} devoted to random field representations for stochastic elliptic boundary value problems and stochastic homogenization. We propose a novel probabilistic modeling  to take into account uncertainties in the spectral measure of the elasticity random  field and we analyze the stochastic elliptic boundary value problem (BVP) that has to be solved to perform the stochastic homogenization.\\

\noindent{\textbf{Notations}

\noindent The following notations are used:\\
$x$: lower-case Latin or Greek letters are deterministic real variables.\\
$\bfx$: boldface lower-case Latin, Greek, and calligraphic letters are deterministic vectors.\\
$X$: upper-case Latin, Greek, and calligraphic letters are real-valued random variables.\\
$\bfX$: boldface upper-case Latin or Greek letters are vector-valued random variables.\\
$[x]$: lower-case Latin of Greek letters between brackets are deterministic matrices.\\
$[\bfX]$: boldface upper-case letters between brackets are matrix-valued random variables.\\
$\curC_w,\curC_y,\curC_z,\curC_\xi,\curC_\varphi,\curC_\curS$:  set of values for $\bfw,[y],\bfz,\bfxi,\bfvarphi,\bfcurS$.\\
$\bfcurS$: vector-valued deterministic parameter that controls spectrum uncertainties.\\
$\CC$: fourth-order tensor-valued random field.\\
$\NN$, $\RR$: set of all the integers $\{0,1,2,\ldots\}$, set of all the real numbers.\\
$\RR^n$: Euclidean vector space on $\RR$ of dimension $n$.\\
$\MM_{n,m}$: set of all the $(n\times m)$ real matrices.\\
$\MM_n$: set of all the square $(n\times n)$ real matrices.\\
$\MM_n^S$: set of all the symmetric $(n\times n)$ real matrices.\\
$\MM_n^+$: set of all the positive-definite symmetric $(n\times n)$ real matrices.\\
$[I_{n}]$: identity matrix in $\MM_n$.\\
$\bfx = (x_1,\ldots,x_n)$: point in $\RR^n$.\\
$\langle \bfx,\bfy \rangle_2 = x_1 y_1 + \ldots + x_n y_n$: inner product in $\RR^n$.\\
$\Vert\bfx\Vert_2$:  norm in $\RR^n$ such that $\Vert\bfx\Vert^2 = \langle \bfx,\bfx \rangle_2$.\\
$\Vert\,[a]\,\Vert_2 \, = \, \sup_\bfx\{ \Vert[a]\bfx\Vert_2 /\Vert\bfx\Vert_2\}$ for $[a]\in\MM_n$ and $\bfx\in\RR^n$.\\
$[x]^T$: transpose of matrix $[x]$.\\
$\tr \{[x]\}$: trace of the square matrix $[x]$.\\
$\langle[x] , [y] \rangle_F = \tr \{[x]^T \, [y]\}$, inner product of matrices $[x]$ and $[y]$ in $\MM_{n,m}$.\\
$\Vert\, [x]\, \Vert_F$: Frobenius norm of matrix  $[x]$ such that  $\Vert [x] \Vert^2_F = \langle [x] , [x]\rangle_F$.\\
$\11_B$: indicatrix function of set $B$.\\
$\iota$: imaginary unit.\\
$\delta_{kk'}$: Kronecker's symbol.\\
$\delta_{\bfx_0}$: Dirac measure at point $\bfx_0$.\\
$a.s$: almost surely.\\
$E$: mathematical expectation.\\

\section{Non-Gaussian random field with uncertain spectral measure}
\label{sec:Section2}
\noindent \textit{Physical framework of the considered random fields class}.
For stochastic homogenization of linear elastic heterogeneous media presented in Section~\ref{sec:Section5}, the physical space $\RR^3$ is referred to a Cartesian reference system for which the generic point is $\bfx=(x_1,x_2,x_3)$. Nevertheless, all the developments presented in Sections~\ref{sec:Section2} to \ref{sec:Section4}, can easily be adapted to any finite dimension greater or equal to $1$. We consider a linear elastic heterogeneous medium for which the  elasticity field is a non-Gaussian fourth-order tensor-valued random field $\widetilde\CC =\{\widetilde\CC_{ijpq}\}_{ijpq}$ with $i$, $j$, $p$, and $q$ in $\{1,2,3\}$.
A general probabilistic construction has been proposed in \cite{Guilleminot2013,Guilleminot2013a,Soize2017b} in order to take into account the material symmetry in a given symmetry class for the mean value of the elasticity random field and considering the statistical fluctuations either in the same symmetry class, or in another symmetry class, or in a mixture of two symmetry classes.
In this paper, we start the construction of the random field with the initial formulation proposed in \cite{Soize2006}. It is thus assumed that the mean value of the elasticity random field $\underline{\widetilde\CC}$ is independent of $\bfx$ and belongs to any symmetry class. The statistical fluctuations of $\widetilde\CC$ around $\underline{\widetilde\CC}$ are assumed to be anisotropic and  statistically homogeneous in $\RR^3$ (it should be noted that the developments presented could be extended to a more general case of material symmetry for the statistical fluctuations but would greatly complicate the presentation). An important quantity that controls the statistical fluctuations is the spectral measure that allows the spatial correlation structure to be described (see for instance \cite{Skorokhod1973,Kree1986,Vanmarcke2010,Leonenko2013,Malyarenko2014}) and that we will assumed to be uncertain in this paper.\\

\noindent \textit{Principle of construction of the uncertain spectral measure}.
The uncertainties are modeled using the probability theory. A parameterization of the spectral measure, involving a parameter $\bfcurS$, is introduced. The uncertain spectral measure is obtained by modeling $\bfcurS$ by a random variable $\bfS$. We then construct a non-Gaussian positive-definite fourth-order tensor-valued homogeneous random field $\CC(\cdot;\bfcurS)$ parameterized with $\bfcurS$ such that $\widetilde\CC = \CC(\cdot;\bfS)$. Throughout the paper the quantities surmounted by a tilde correspond to the case of uncertain spectral measure modeled by a random spectral measure.\\

\noindent \textit{Non-Gaussian positive-definite matrix-valued homogeneous random fields with uncertain spectral measure}. For all $\bfx$ fixed in $\RR^3$,
the fourth-order random tensor $\CC(\bfx;\bfcurS)$ verifies the usual symmetry and positiveness properties. Let $\textbf{i} =(i,j)$ with $1\leq i\leq j\leq 3 $ and $\textbf{j} =(p,q)$ with $1\leq p\leq q\leq 3$ be the indices with values in $\{1,\ldots , 6\}$, which allow for defining the $\MM^+_6$-valued random matrix $[\CC(\bfx;\bfcurS)]$ such that $[\CC(\bfx;\bfcurS)]_{\textbf{i}\textbf{j}} = \CC_{ijpq}(\bfx;\bfcurS)$ (use of the representation in Voigt notation for the constitutive equation).\\

\noindent \textit{Random effective elasticity matrix}. For fixed $\bfcurS$, the parameterized effective elasticity matrix $[\CC^\peff(\bfcurS)]$ is a random matrix in $\MM^+_6$,
which is obtained by stochastic homogenization solving a stochastic elliptic BVP on a bounded domain $\Omega$ of $\RR^3$. The random effective elasticity matrix $[\widetilde\CC^\peff]$, corresponding to the elasticity random field with uncertain spectral measure, is then given by $[\widetilde\CC^\peff] = [\CC^\peff(\bfS)]$.
%
%---DEFINITION 1 ----------------
\begin{definition}[{\textit{Non-Gaussian homogeneous random field}} $\hbox{[}\CC(\cdot;\bfcurS)\hbox{]}$ {\textit{given}} $\bfcurS$] \label{definition:1}
Let $[\underline\CC]$ be a given matrix in $\MM^+_6$ independent of $\bfx$ and $\bfcurS$. We define $\{ [\CC(\bfx;\bfcurS)],\bfx\in\RR^3\}$ as a non-Gaussian $\MM^+_6$-valued second-order random field, on a probability space $(\Theta,\curT,\curP)$, indexed by $\RR^3$, homogeneous, mean-square continuous, whose mean value is
$[\underline\CC] = E\{[\CC(\bfx;\bfcurS)]\}$ that is therefore independent of $\bfx$ and $\bfcurS$. We have
\begin{equation}\label{eq:eq1}
\tr \,[\underline\CC] = \underline c_1 \quad , \quad \langle [\underline\CC]\bfomega , \bfomega\rangle_2 \,\,\, \geq \,\, \underline c_0\, \Vert\bfomega\Vert_2^2  \,\, ,\,\, \forall \,\bfomega\in\RR^6\, ,
\end{equation}
in which $\underline c_0$ and $\underline c_1$ are two positive finite constants.
\end{definition}
With the construction proposed in this paper, $E\{[\widetilde\CC(\bfx)]\} = E\{[\CC(\bfx;\bfS)]\}$ will not be equal to $[\underline\CC]$ (that is not a difficulty). However, we will see that $E\{[\widetilde\CC(\bfx)]\} \simeq [\underline\CC]$.
%
%---LEMMA 1 ----------------------
\begin{lemma}[Normalization of random field $\hbox{[}\CC(\cdot;\bfcurS)\hbox{]}$ given $\bfcurS$] \label{lemma:1}
Let $[\underline\LL]$ be the upper triangular $(6\times 6)$ real matrix such that $[\underline\CC]= [\underline\LL]^T\, [\underline\LL]$. For fixed $\bfcurS$, the normalized representation of $[\CC(\bfx;\bfcurS)]$ is written as,
\begin{equation} \label{eq:eq2}
[\CC(\bfx;\bfcurS)] = \frac{1}{1+\epsilon}\,[\underline\LL]^T(\epsilon \,[I_6] + [\bfC(\bfx;\bfcurS)] ) [\underline\LL]\, ,
\end{equation}
in which $\epsilon > 0$ is given and where $\{ [\bfC(\bfx;\bfcurS)],\bfx\in\RR^3\}$ is a $\MM^+_6$-valued random field (by construction), defined on $(\Theta,\curT,\curP)$,
indexed by $\RR^3$. Then $ [\bfC(\cdot;\bfcurS)]$ is homogeneous, mean-square continuous, and such that
\begin{equation} \label{eq:eq3}
E\{[\bfC(\bfx;\bfcurS)]\} = [I_6] \quad , \quad \forall\,\bfx \in \RR^3\, .
\end{equation}
\end{lemma}
%
%------ PROOF LEMMA 1 ------------------
\begin{proof} (Lemma~\ref{lemma:1}).
Under the hypotheses introduced in Definition~\ref{definition:1}, it is easy to see that $[\bfC(\cdot;\bfcurS)]$ is a second-order, homogeneous, mean-square continuous random field, and satisfies Eq.~\eqref{eq:eq3}.
\end{proof}
It should be noted that the lower bound $\epsilon\,[\underline\CC]/(1+\epsilon)$ used in Eq.~\eqref{eq:eq2} could be replaced by a more general lower bound $[\CC_b]\in\MM_6^+$ as proposed in \cite{Soize2017b,Soize2017e}. Note also that, as previously, introducing $[\widetilde\bfC(\bfx)\} = [\bfC(\bfx;\bfS)]$, $E\{[\widetilde\bfC(\bfx)]\}$ will not be equal to $[I_6]$ and we will see that $E\{[\widetilde\bfC(\bfx)]\} \simeq [I_6]$.
%
%--- HYPOTHESIS 1 --------------------
\begin{hypothesis}[{\textit{Principle of construction of random field}} $\hbox{[}\bfC(\cdot;\bfcurS)\hbox{]} $ {\textit{given}} $\bfcurS$] \label{hypothesis:1}
By construction (see Lemma~\ref{lemma:1}), $[\bfC(\cdot;\bfcurS)]$ is a $\MM^+_6$-valued random field indexed by $\RR^3$ and homogeneous.
For all $\bfx$ fixed in $\RR^3$, the $\MM^+_6$-valued random variable $[\bfC(\bfx;\bfcurS)]$ is constructed by using the Maximum Entropy Principle under the following available information,
\begin{equation} \label{eq:eq6}
E\{[\bfC(\bfx;\bfcurS)]\} = [I_6] \quad , \quad E\{ \log (\det[\bfC(\bfx;\bfcurS)] ) \} = b_c \, ,
\end{equation}
in which $b_c$ is independent of $\bfx$ and $\bfcurS$ and such that $\vert b_c \vert < +\infty$. The second equality is introduced in order  that the random matrix
$[\bfC(\bfx;\bfcurS)]^{-1}$ (that exists almost surely) be a second-order random variable:
$E\{\Vert[\bfC(\bfx;\bfcurS)]^{-1}\Vert_2^2\} \leq E\{\Vert[\bfC(\bfx;\bfcurS)]^{-1}\Vert_F^2\} < +\infty$.
With such a construction, $[\bfC(\bfx;\bfcurS)]$ will appear as a nonlinear transformation of $6\times(6+1)/2 = 21$ independent normalized Gaussian real-valued random variables $\{ G_{mn}(\bfx;\bfcurS), 1 \leq m \leq n\leq 6\}$, such that
\begin{equation} \label{eq:eq8}
E\{G_{mn}(\bfx;\bfcurS)\} = 0 \quad , \quad E\{G_{mn}(\bfx;\bfcurS)^2\} = 1 \, .
\end{equation}
The spatial correlation structure of random field $\{[\bfC(\bfx;\bfcurS)],\bfx\in\RR^3\}$ is introduced by considering  $21$ independent  real-valued random fields
$\{G_{mn}(\bfx;\bfcurS),\bfx\in\RR^3\}$  for $1\leq m \leq n \leq 6$,  corresponding to $21$ independent copies of a unique normalized Gaussian homogeneous mean-square continuous real-valued random field $\{G(\bfx;\bfcurS),\bfx\in\RR^3\}$  given  its normalized spectral measure parameterized by $\bfcurS$.
Note that the Gaussian random field $G(\cdot;\bfcurS)$ is entirely defined by its normalized spectral measure (parameterized by $\bfcurS$) because, $\forall\bfx\in\RR^3$,
$E\{G(\bfx;\bfcurS)\} = 0$ and $E\{G(\bfx;\bfcurS)^2\} = 1$.
The constant $b_c$ is eliminated in favor of a hyperparameter $\delta_c$,  which allows for controlling the level of statistical fluctuations of $[\bfC(\bfx;\bfcurS)]$, defined by $\delta_c =(E\{\Vert [\bfC(\bfx;\bfcurS)] - [I_6]\Vert_F^2\} / 6)^{1/2}$, independent of $\bfx$ and chosen independent of $\bfcurS$.
\end{hypothesis}
%
%---- PROPOSITION 1 ----------------------
\begin{proposition}[Random field $\hbox{[}\bfC(\cdot;\bfcurS)\hbox{]}$] \label{proposition:1}
Let us assume Hypothesis~\ref{hypothesis:1}.

\noindent (i) Let $d^S C = 2^{15/2} \Pi_{1\leq m\leq n\leq 6} \,dC_{mn}$ be the volume element on Euclidean space $\MM^S_6$ in which $dC_{mn}$ is the Lebesgue measure on $\RR$. For all $\bfx$ fixed in $\RR^3$, the probability measure $P_{[\bfC(\bfx;\bfcurS)]}(d^S C)$ of the $\MM^+_6$-valued random variable $[\bfC(\bfx;\bfcurS)]$ constructed with the Maximum Entropy Principle under the constraints defined by Eq.~\eqref{eq:eq6}, is independent of $\bfx$ (homogeneous random field), independent of $\bfcurS$ (this marginal probability measure does not depend of the correlation structure), and written as $P_{[\bfC(\bfx;\bfcurS)]}(d^S C) = p_{[\bfC(\bfx;\bfcurS)]}([C])\,d^S C$ in which the probability density function is written as
$p_{[\bfC(\bfx;\bfcurS)]}([C]) = \11_{\MM^+_6}([C])\, c_c (\det [C])^{7(1-\delta_c^2)/(2\delta_c^2)}\exp(-7 \tr\{[C]\} /(2\delta_c^2))$
with $c_c$ the normalization constant and where hyperparameter $\delta_c$ must belong to the real interval $]0\, , \sqrt{7/11}[$.

\noindent (ii) For all $\bfx$ fixed in $\RR^3$, random matrix $[\bfC(\bfx;\bfcurS)]$ is written as
\begin{equation} \label{eq:eq11}
[\bfC(\bfx;\bfcurS)] = [\bfL(\bfx;\bfcurS)]^T\, [\bfL(\bfx;\bfcurS)] \, ,
\end{equation}
in which $[\bfL(\bfx;\bfcurS)]$ is an upper triangular random matrix in $\MM_6$ such that

\noindent 1) the $21$ random variables $\{[\bfL(\bfx;\bfcurS)]_{mn}, 1\leq m\leq n\leq 6\}$ are mutually independent.

\noindent 2) for $1\leq m < n\leq 6$, $[\bfL(\bfx;\bfcurS)]_{mn} = \sigma_c\, G_{mn}(\bfx;\bfcurS)$ with $\sigma_c= \delta_c/\sqrt{7}$ and where $G_{mn}(\bfx;\bfcurS)$
is a normalized Gaussian real-valued random variable (see Eq.~\eqref{eq:eq8}).

\noindent 3) for $1\leq m=n\leq 6$, $[\bfL(\bfx;\bfcurS)]_{mm} = \sigma_c\, \sqrt{ 2\, h( G_{mm}(\bfx;\bfcurS);\alpha_m) }$ with
$\alpha_m = 1/ (2\sigma_c^2) + (1-m)/2$  such that $\alpha_1 > \ldots > \alpha_6 > 3$ and where $G_{mm}(\bfx;\bfcurS)$ is a normalized Gaussian real-valued random variable (see Eq.~\eqref{eq:eq8}). The function $b\mapsto h(b;\alpha)$ is such that $\GammaGamma_\alpha = h(\curN;\alpha)$ is a Gamma random variable with parameter $\alpha$
when $\curN$ is the normalized Gaussian real-valued random variable.

\noindent (iii) The $21$ random fields $\{G_{mn}(\bfx;\bfcurS),\bfx\in\RR^3\}$  for $1\leq m \leq n \leq 6$ are $21$ independent copies of a  normalized Gaussian homogeneous mean-square continuous real-valued random field $\{G(\bfx;\bfcurS),\bfx\in\RR^3\}$,
\begin{equation} \label{eq:eq12}
E\{G(\bfx;\bfcurS)\} = 0 \quad , \quad E\{G(\bfx;\bfcurS)^2\} = 1 \quad , \quad \forall \bfx\in\RR^3 \, ,
\end{equation}
and which will be defined in Section~\ref{sec:Section3} for imposing its spatial correlation structure via its spectral measure. parameterized by $\bfcurS$.
\end{proposition}
%
%------ PROOF OF PROPOSITION 1 ------------------
\begin{proof} (Proposition~\ref{proposition:1}).
We refer the reader  to \cite{Soize2000} for the construction using the Maximum Entropy Principle and to \cite{Soize2006,Soize2017e,Soize2017b} for the representation  defined by Eq.~\eqref{eq:eq11}.
However, we have to prove the properties that yield $E\{[\bfC(\bfx;\bfcurS)]\} = [I_6]$, this proof being used in Remark~\ref{remark:1}.
For $1\leq m\leq n\leq 6$, $[\bfC(\bfx;\bfcurS)]_{mn} = \sum_{\ell=1}^m [\bfL(\bfx;\bfcurS)]_{\ell m}\, [\bfL(\bfx;\bfcurS)]_{\ell n}$.
For $m=n$, $E\{[\bfC(\bfx;\bfcurS)]_{mm} = 2 \sigma_c^2\, E\{h( G_{mm}(\bfx;\bfcurS);\alpha_m)\} + \sigma_c^2\sum_{\ell < m} E\{ G_{\ell m}(\bfx;\bfcurS)^2\}$.
Eq.~\eqref{eq:eq8} yields $\sum_{\ell < m} E\{ G_{\ell m}(\bfx;\bfcurS)^2\}$ $= \sum_{\ell < m} 1 = m-1$ and
$E\{ h( G_{mm}(\bfx;\bfcurS);\alpha_m)\} = E\{ \GammaGamma_{\alpha_m}\} = \alpha_m = 1/(2\sigma_c^2) + (1-m)/2$. Therefore,
$E\{ [\bfC(\bfx;\bfcurS)]_{mm}\} = 1$.
For $1 \leq m < n \leq 6$, we have $E \{ [\bfC(\bfx;\bfcurS)]_{mn} = \sigma_c^2\, E\{ G_{mn}(\bfx;\bfcurS)$ $\sqrt{2\, h( G_{mm}(\bfx;\bfcurS);\alpha_m)}\} + \sigma_c^2\sum_{\ell < m} E\{ G_{\ell m}(\bfx;\bfcurS)\,$ $G_{\ell n}(\bfx;\bfcurS)\}$.
Proposition~\ref{proposition:1}-(iii) and Eq.~\eqref{eq:eq8} yield
$E\{ G_{\ell m}(\bfx;\bfcurS)$ $\,G_{\ell n}(\bfx;\bfcurS)\}\! = \! E\{ G_{\ell m}(\bfx;\bfcurS)\}\times E\{G_{\ell n}(\bfx;\bfcurS)\} = 0$ and
$E \{ G_{mn}(\bfx;\bfcurS)\sqrt{ 2\, h( G_{mm}(\bfx;\bfcurS);\alpha_m) } \} $ $= E\{ G_{mn}(\bfx;\bfcurS)\}\times E\{ \sqrt{2\, h( G_{mm}(\bfx;\bfcurS);\alpha_m)}\} \! = \! 0$. Therefore, $E\{ [\bfC(\bfx;\bfcurS)]_{mn}\} = 0$. We thus obtain the first equation in Eq.~\eqref{eq:eq6}.
\end{proof}
%
%---LEMMA 2 ----------------------
\begin{lemma}[Properties of function $h$] \label{lemma:2}
(i) Let be $\alpha > 3$. Function $b\mapsto h(b;\alpha): \RR\mapsto ]0\, , +\infty[$ defined in Proposition~\ref{proposition:1} is written as
$h(b;\alpha) = F_\alpha^{-1}(F(b))$ in which $F(b) = \int_{-\infty}^b (2\pi)^{-1/2} e^{-t^2/2}\, dt$ and where $F_\alpha^{-1}$ is the reciprocical function of
$F_\alpha(\gammagamma) = \GammaGamma(\alpha)^{-1}\times$ $\int_0^\gammagamma t^{\alpha -1} e^{-t}\, dt$ for $\gammagamma \geq 0$ with $\GammaGamma(\alpha) = \int_0^{+\infty} t^{\alpha -1} e^{-t}\, dt$. For all $\alpha > 3$ and for all $b\in\RR$, $h(b;\alpha) \leq 2\alpha + b^2$.

\noindent (ii) If $\curN$ is  Gaussian with $E\{\curN\} = 0$ and $E\{\curN^2\} = 1$, then $E\{ h(\curN;\alpha) \} = \alpha$.

\noindent (iii) If $\curG$ is non-Gaussian with $E\{\curG\} = 0$ and $E\{\curG^2\} = 1$, we have $E\{ h(\curG;\alpha) \} \not = \alpha$, but
for $\alpha\rightarrow +\infty$, $E\{ h(\curG;\alpha)\} \rightarrow \alpha$.
\end{lemma}
%
%
%------ PROOF LEMMA 2 ------------------
\begin{proof} (Lemma~\ref{lemma:2}).
(i) Function $h$ defined in Proposition~\ref{proposition:1} shows that $F_\alpha(h(b;\alpha)) = F(b)$ in which $F(b)$ is the c.d.f of $\curN$ and $F_\alpha(\gammagamma)$ is the c.d.f of the $\GammaGamma_\alpha$ random variable such that $E\{\GammaGamma_\alpha\} = \alpha$. We then deduce that $h(b;\alpha) = F_\alpha^{-1}(F(b))$.
In order to prove that $h(b;\alpha) \leq 2\alpha + b^2$, since $u\mapsto F_\alpha^{-1}(u)$ is a strictly increasing function from $[0,1[$ into $\RR^+$, we have to prove that, for all $\alpha > 3$ and for all $b\in\RR$, we have $J_\alpha(b) \geq 0$ in which
\begin{equation} \label{eq:eq13}
J_\alpha(b) = F_\alpha(2\alpha+b^2)- F(b) = \frac{1}{\GammaGamma(\alpha)}\int_0^{2\alpha+b^2} t^{\alpha-1}e^{-t}\, dt -\frac{1}{\sqrt{2\pi}}\int_{-\infty}^b e^{-t^2/2}\, dt\, ,
\end{equation}
which can be rewritten as
\begin{equation} \label{eq:eq14}
J_\alpha(b) = P_\alpha + \frac{1}{\GammaGamma(\alpha)}\int_{2\alpha}^{2\alpha+b^2} t^{\alpha-1} e^{-t}\, dt -\frac{\sgn(b)}{\sqrt{2\pi}}\int_0^{\vert b\vert} e^{-t^2/2}\, dt\, ,
\end{equation}
in which $P_\alpha = F_\alpha(2\alpha)-1/2$ that is such that, $\forall\alpha > 3$, $P_\alpha > F_3(6)-1/2 \simeq 0.438$, and $\sgn$ is the sign function.
Equation~\eqref{eq:eq14} shows that $\forall\alpha > 3$  and $\forall \, b < 0$, $J_\alpha(b) > 0$ because $J_\alpha(b)$ is the sum of three positive terms. For $b=0$ and
$\forall\alpha > 3$, Eq~.~\eqref{eq:eq14} shows that $J_\alpha(0) = P_\alpha > 0$. For $b> 0$, we use Eq~.~\eqref{eq:eq13}.
Let $J'_\alpha(b)= d J_\alpha(b)/db$ be such that
$J'_\alpha(b) = (2\pi)^{-1/2} e^{-b^2/2}  (a_\alpha  b \,(1+{b^2}/{(2\alpha)})^{\alpha-1} e^{-b^2/2} - 1 )$,
in which $a_\alpha = 2\sqrt{2\pi} e^{-2\alpha} (2\alpha)^{\alpha-1} /\GammaGamma(\alpha)$ that is such that, $\forall\alpha > 3$, $a_\alpha \leq 0.223679$.
From Eq.~\eqref{eq:eq13}, it can be seen  that, $\forall\alpha > 3$, for $b\rightarrow +\infty$, $J_\alpha(b) \rightarrow 0$ and
$J'_\alpha(b) \sim -e^{-b^2/2}/\sqrt{2\pi}$. Consequently, $J_\alpha \rightarrow 0_+$. Since $J_\alpha(0) > 0$, we will have $J_\alpha(b) > 0$ for $b > 0$ if
$J'_\alpha(b) < 0$, that is to say if $a_\alpha b\, (1+\frac{b^2}{2\alpha})^{\alpha-1} < e^{b^2/2}$, which is true for $\alpha > 3$ and $b >0$.
(ii) If $\curN$ is normalized and Gaussian, then $h(\curN;\alpha) = \GammaGamma_\alpha$ and hence $E\{h(\curN;\alpha)\}=\alpha$.
(iii)  Let $P(x^2 \vert \widehat\alpha)$ be defined by
$P(x^2 \vert \widehat\alpha) = ( 2^{\widehat\alpha/2} \GammaGamma(\widehat\alpha/2) )^{-1} \int_0^{x^2} v^{(\widehat\alpha/2) -1} e^{-v/2} \, dv$.
The change of variable $v=2t$ yields
$P(x^2 \vert \widehat\alpha) = ( \GammaGamma(\widehat\alpha/2) )^{-1} \int_0^{x^2/2} t^{(\widehat\alpha/2) -1} e^{-t} \, dt$. Taking $\widehat\alpha = 2\alpha$ yields
$P(x^2 \vert \widehat\alpha) = F_\alpha(x^2/2)$. For $\widehat\alpha \rightarrow +\infty$, $P(x^2 \vert \widehat\alpha) \sim F(y)$ with $y=(x^2-\widehat\alpha)/\sqrt{2\widehat\alpha}$ and hence $F_\alpha(x^2/2) \sim F(y)$. Since $F_\alpha(h(b;\alpha))=F(b)$, taking $x^2/2 = h(b;\alpha)$ yields
$F(b)\sim F(y)$ that is to say $b \sim y = (2h(b;\alpha) - 2\alpha ) /\sqrt{4\alpha}$, which shows that $h(b;\alpha) \sim \alpha + b\sqrt{\alpha}$. If $\curG$ is a non-Gaussian random variable such that $E\{\curG\}=0$, then $E\{h(\curG;\alpha)\} \sim \alpha + \sqrt{\alpha} \, E\{\curG\} = \alpha$.
\end{proof}

%----- REMARK 1 ------------------------------
\begin{remark} \label{remark:1}
The analysis of the proof of Proposition~\ref{proposition:1} shows that the property $E\{[\bfC(\bfx;\bfcurS)]\} = [I_6]$ holds for a fixed value of $\bfcurS$. When $\bfcurS$ will be modeled by a random variable $\bfS$ in order to take into account uncertainties in the spectral measure (see Section~\ref{sec:Section4} and as we have previously explained  at the beginning of Section~\ref{sec:Section2}) the random field $\{[\widetilde\bfC(\bfx)],\bfx\in\RR^3\}$ such that $[\widetilde\bfC(\bfx)] = [\bfC(\bfx;\bfS)]$ will then depend on copies of the random field $\{G(\bfx;\bfS),\bfx\in\RR^3\}$ that will always satisfy $E\{G(\bfx;\bfS)\} = 0$ and $E\{G(\bfx;\bfS)^2\} = 1$, but which will no longer be Gaussian. By examining the proof of Proposition~\ref{proposition:1}, it can be seen that we will always have
$E\{[\widetilde\bfC(\bfx)]_{mn}\}=0$ for $m\not= n$ but that $E\{[\widetilde\bfC(\bfx)]_{mm}\}= 2 \, \sigma_c^2 \, E\{h(G(\bfx;\bfS);\alpha_m)\} + \sigma_c^2(m-1) \not = 1$. Nevertheless, from Proposition~\ref{proposition:1}, since $\alpha_m > 3$ and from Lemma~\ref{lemma:2}-(iii), taking $\curG = G(\bfx;\bfS)$ yields
$E\{[\widetilde\bfC(\bfx)]_{mm}\}\simeq 1$. It has numerically been verified that, if $\curG$ is a uniform random variable (that will not be the case for $G(\bfx;\bfS)$) such that $E\{\curG\} = 0$ and $E\{\curG^2\} = 1$ (that will  be the case for $G(\bfx;\bfS)$), then we have $E\{h(\curG;\alpha_m)\} \simeq \alpha_m$ with an error of  $5\times 10^{-4}$ for all $1\leq m\leq 6$.
\end{remark}

\section{Construction and analysis of the Gaussian random field $G(\cdot;\bfcurS)$ with uncertain spectral measure parameterized by $\bfcurS$}
\label{sec:Section3}
We start by constructing a normalized Gaussian, homogeneous, second-order, mean-square continuous random field $\{G(\bfx),\bfx\in\RR^3\}$. This field corresponds to $G(.;\underline\bfcurS)$ for which its spectral measure is given  and represented by a given value $\underline\bfcurS$ of $\bfcurS$, that is to say, $G = G(\cdot;\underline\bfcurS)$.
Therefore, there exists a positive bounded spectral measure $m_G(d\bfk)$ on $\RR^3$ such that the correlation function $\rho_G$ of $G$ is written,
for all $\bfx$ and $\bfzeta$ in $\RR^3$, as
\begin{equation} \label{eq:eq19}
\rho_G(\bfzeta) = E\{G(\bfx+\bfzeta)\, G(\bfx)\} = \int_{\RR^3} e^{\iota\, \bfk\cdot\bfzeta} m_G(d\bfk) =  \int_{\RR^3} \cos(\bfk\cdot\bfzeta)\, m_G(d\bfk)\, ,
\end{equation}
in which $\iota =\sqrt{-1}$, $\bfk\cdot\bfzeta = \sum_{j=1}^3 k_j\zeta_j$, $d\bfk = dk_1dk_2dk_3$. In addition, it is assumed that
$m_G(d\bfk) = s(\bfk)\, d\bfk$ admits a spectral density function (s.d.f) $\bfk\mapsto s(\bfk): \RR^3\rightarrow \RR^+$. Equations~\eqref{eq:eq12} and \eqref{eq:eq19}
yield $\rho_G(0)= E\{G(\bfx)^2\} = 1$, and consequently,
\begin{equation} \label{eq:eq20}
  \int_{\RR^3} s(\bfk)\,d\bfk = 1\, .
\end{equation}
In this section, we begin with the analysis of the spectral measure and the modeling of uncertainties. Then we introduce a finite representation of $G(\cdot;\bfcurS)$ and we study its properties. Note that a dimensionless s.d.f, $\chi$, of $s$ will be introduced.
%
%--- HYPOTHESIS 2 --------------------
\begin{hypothesis}[{\textit{Spectral density function}} $s$ {\textit{and spatial correlation length of}} $G$] \label{hypothesis:2}
It is assumed that $s$ has a compact support $\overline\KK = \partial\KK\cup\KK$ with
$\KK= \Pi_{j=1}^3 \,  ]-K_j , K_j[$ in which $K_j\in [K_j^\pmin, K_j^\pmax]$ with $0 < K_j^\pmin < K_j^\pmax < +\infty$.
It is assumed that $s$ is a continuous function on $\RR^3$. Since supp $s = \overline\KK$, we must have $s(\bfk)=0$, $\forall\bfk\in\partial\KK$, and thus
\begin{equation} \label{eq:eq23}
\int_{\RR^3} s(\bfk)\, d\bfk =\int_{\overline\KK} s(\bfk)\, d\bfk = 1\, .
\end{equation}
Since $G$ is real, we have $s(-\bfk) = s(\bfk)$ for all $\bfk\in\RR^3$. In addition to this symmetry property,
we assume that $s$ satisfies the following quadrant symmetry \cite{Vanmarcke2010}: defining $\bfk_{\{-j\}}$ as vector $\bfk$ for which its component $k_j$ is replaced by $-k_j$, then $s(\bfk_{\{-j\}}) = s(\bfk)$ for $j=1,2,3$ and $\forall\bfk\in\RR^3$. The spatial correlation length for coordinate $\zeta_j$ is defined by
\begin{equation} \label{eq:eq23bis}
L_{cj} = \int_0^{+\infty} \vert \rho_j(0,\ldots,\zeta_j,\ldots,0)\vert \, d\zeta_j \, ,
\end{equation}
and is assumed to be finite.
\end{hypothesis}
%
%---DEFINITION 2 ----------------
\begin{definition}[{\textit{Spectral domain sampling}}] \label{definition:2}
Let $\nu_s$ be a given even integer.
For $j\in\{1,2,3\}$, we define $\Delta_j = 2\,K_j/\nu_s$ as the sampling step of interval $[-K_j,K_j]$
and $k_{j\beta_j} = -K_j +( \beta_j-1/2)\,\Delta_j$ for $\beta_j=1,\ldots , \nu_s$ as its spectral sampling points.
Let $\curB = \{\bfbeta = (\beta_1,\beta_2,\beta_3)\, , \beta_j =1,\ldots , \nu_s\}$ be the finite subset of  $\NN^3$.
We define $\Delta$, $K$, $\nu$, and $\bfk_\bfbeta$ such that: $\Delta =\Delta_1\Delta_2\Delta_3$, $K=K_1K_2K_3$, $\nu=(\nu_s)^3$, and for all $\beta\in\curB$, $\bfk_\beta=(k_{1\beta_1},k_{2\beta_2},k_{3\beta_3})\in\KK\subset\RR^3$.
\end{definition}
%
%
%
%---LEMMA 4 ----------------------
\begin{lemma}[Discretization of the spectral measure and convergence properties] \label{lemma:4}
Let $\delta_{\bfk_\bfbeta}(\bfk) = \otimes_{j=1}^3 \delta_{k_{j\beta_j}}(k_j)$ be the Dirac measure on $\RR^3$ at sampling point $\bfk_\bfbeta\in\KK\subset \RR^3$ defined in Definition~\ref{definition:2}. Let $m_G^\nu(d\bfk)$ be the positive bounded measure on $\RR^3$ defined by
\begin{equation} \label{eq:eq30}
m_G^\nu(d\bfk) = \sum_{\bfbeta\in\curB} s_\bfbeta^\Delta \,  \delta_{\bfk_\bfbeta}(\bfk) \quad , \quad  s_\bfbeta^\Delta  =\Delta\, s(\bfk_\bfbeta)\, ,
\end{equation}
which is such that $m_G^\nu(\RR^3) =\sum_{\bfbeta\in\curB} s_\bfbeta^\Delta =\eta_\nu$ with $\eta_\nu > 0$. The sequence of measures $\{m_G^\nu(d\bfk)\}_\nu$ converges narrowly towards the measure $m_G(d\bfk)$ and the   positive sequence $\{\eta_\nu\}_\nu$ converges towards $1$.
\end{lemma}
%
%
%------ PROOF LEMMA 4 ------------------
\begin{proof} (Lemma~\ref{lemma:4}). We have to prove that $\forall f\in C^0(\overline\KK)$, the sequence
$m_G^\nu(f) = \int_{\overline\KK} f(\bfk)\,$ $m_G^\nu(d\bfk) =\Delta \sum_{\bfbeta\in\curB} f(\bfk_\bfbeta)\, s(\bfk_\bfbeta)$
converges towards $m_G(f)=\int_{\overline\KK} f(\bfk)\,m_G(d\bfk)= \int_{\overline\KK} f(\bfk)$ $s(\bfk)\,d\bfk$.
Since the function $\bfk\mapsto f(\bfk)\, s(\bfk)$ is continuous on $\overline\KK$, it is known that for $\nu_s\rightarrow +\infty$ (that is to say for
$\nu\rightarrow +\infty$), $\Delta\sum_{\bfbeta\in\curB} f(\bfk_\bfbeta)\, s(\bfk_\bfbeta) \rightarrow \int_{\overline\KK} f(\bfk)\,s(\bfk)\,d\bfk$.
Taking $f(\bfk)=1$ for all $\bfk\in\overline\KK$ yields $m_G(1)=\int_{\overline\KK} s(\bfk)\,d\bfk = 1$ and $m_G^\nu(1)=\Delta \sum_{\bfbeta\in\curB}s(\bfk_\bfbeta)
=\sum_{\bfbeta\in\curB} s_\bfbeta^\Delta =\eta_\nu$. Therefore, $\{\eta_\nu\}_\nu$ converges towards $1$.
\end{proof}
%
%--- HYPOTHESIS 3 --------------------
\begin{hypothesis}[{\textit{Choice of}} $\nu=(\nu_s)^3$] \label{hypothesis:3}
Let $s$ be the s.d.f satisfying Hypothesis~\ref{hypothesis:2}. Let us consider the spectral domain sampling introduced in Definition~\ref{definition:2}. Using Lemma~\ref{lemma:4}, we will assume that $\nu$ is chosen sufficiently large in order that $\vert \sum_{\bfbeta\in\curB} s_\bfbeta^\Delta - 1\vert \leq \epsilon_s \ll 1$
and consequently, we will write $\sum_{\bfbeta\in\curB} s_\bfbeta^\Delta \simeq 1$.
\end{hypothesis}
%
%
%---DEFINITION 3 ----------------
\begin{definition}[{\textit{Dimensionless spectral density function}}] \label{definition:3}
The spectral density function $\bfk\mapsto s(\bfk)$, which verifies Hypothesis~\ref{hypothesis:2}, is written for all $\bfk$ in $\RR^3$ as
$s(k_1,k_2,k_3) = (K_1 K_2 K_3)^{-1}\chi(k_1/K_1,k_2/K_2,k_3/K_3)$ in which $\chi$ is a given function $\bftau=(\tau_1,\tau_2,\tau_3)\mapsto \chi(\tau_1,\tau_2,\tau_3):\RR^3 \rightarrow\RR^+$ with compact support $[-1,1]^3$.
\end{definition}
Function $\chi$  has the same properties as $s$: $\chi(-\bftau) = \chi(\bftau)$, quadrant symmetry, and continuity. For $j=1,2,3$, the change of variable $\tau_j = k_j/K_j$ yields $s(\bfk)\, d\bfk = \chi(\bftau)\ d\bftau$ and thus $m_G(d\bfk) = \mu_G(d\bftau)$ with $\mu_G(d\bftau) = \chi(\bftau)\, d\bftau$, and consequently, Eq.~\eqref{eq:eq23} yields
\begin{equation} \label{eq:eq32}
\int_{\RR^3} \chi(\bftau)\, d\bftau =\int_{[-1,1]^3}  \chi(\bftau)\, d\bftau = 1\, .
\end{equation}
The dimensionless spectral domain sampling is directly deduced from Definition~\ref{definition:2},
\begin{equation} \label{eq:eq33}
\{\bftau_\bfbeta = (\tau_{\beta_1},\tau_{\beta_2},\tau_{\beta_3}) , \bfbeta\in\curB\} \,\, , \,\,
\tau_{\beta_j} = -1 +(\beta_j -\frac{1}{2})\,\frac{2}{\nu_s} \,\, , \,\, j\in\{1,2,3\} \, .
\end{equation}
The discretization $\mu_G^\nu(d\bftau)$ of $\mu_G(d\bftau)$, such that $\mu_G^\nu(d\bftau)= m_G^\nu(d\bfk)$, is written as
\begin{equation} \label{eq:eq34}
\mu_G^\nu(d\bftau) = \sum_{\bfbeta\in \curB} \chi^\Delta_\bfbeta \,\, \delta_{\bftau_\bfbeta}(\bftau)\quad , \quad
\chi^\Delta_\bfbeta =({2}/{\nu_s})^3 \, \chi(\bftau_\bfbeta) \, ,
\end{equation}
in which $\delta_{\bftau_\bfbeta} = \otimes_{j=1}^3 \delta_{\tau_{\beta_j}}(\tau_j)$ and where, from Hypothesis~\ref{hypothesis:3},
\begin{equation} \label{eq:eq34plus}
\sum_{\bfbeta\in\curB} \, \chi_\bfbeta^\Delta \simeq 1\, .
\end{equation}
Definition~\ref{definition:3} implies that measure $\mu_G^\nu(d\bftau)$ is independent of $K_1$, $K_2$, and $K_3$.
In order to introduce the probability model of the spectral measure, we start by defining an adapted parameterization $[y]$ that takes into account quadrant symmetry.
%
%---DEFINITION 4 ----------------
\begin{definition}[{\textit{Parameterization of the discretized dimensionless spectral measure}}] \label{definition:4}
Let $\widehat\nu_s = \nu_s/2$ ($\nu_s$ is even).
Let $\curC_y$ be the subset of $\MM_{3,\widehat\nu_s}$ defined by
\begin{equation} \label{eq:eq34bis}
\curC_y = \{ \, [y]\in\MM_{3,\widehat\nu_s} \,\, , \, [y]_{j\widehat\beta} \in [0,1]\,\,\, \hbox{for}  \,\,\, j=1,2,3 \,\,\, \hbox{and} \,\,\, \widehat\beta =1,\ldots , \widehat\nu_s\} \, .
\end{equation}
Let $[\underline y]$ be in $\curC_y$ such that $[\underline y]_{j\widehat\beta} =1/2$ for $j=1,2,3$ and $\widehat\beta =1,\ldots , \widehat\nu_s$.

Let $\widehat\curB = \{\widehat\bfbeta = (\widehat\beta_1,\widehat\beta_2,\widehat\beta_3), \widehat\beta_j = 1,\ldots ,\widehat\nu_s\} \subset \curB$ be the set of
$\widehat\nu = (\widehat\nu_s)^3 = \nu/8$ elements.
We define the finite family  of functions $[y]\mapsto a_{\widehat\bfbeta}([y]) : \curC_y\rightarrow \RR$ such that
$a_{\widehat\bfbeta}([y]) =\sqrt{\chi_{\widehat\bfbeta}^\Delta}\, q_{\widehat\bfbeta}([y];\delta_s)$, in which $\delta_s > 0$ is a hyperparameter that will allow the level of spectrum uncertainties to be controlled and where $[y]\mapsto q_{\widehat\bfbeta}([y];\delta_s)$ is any given continuous real function on $\curC_y$ such that $q_{\widehat\bfbeta}([\underline y];\delta_s) = 1$.  For all $[y]\in\curC_y$, let $\{a_\bfbeta([y]),\bfbeta\in\curB\}$ be the $\nu$ real numbers that are directly constructed from
$\{ a_{\widehat\bfbeta}([y]),\widehat\bfbeta\in\widehat\curB\}$ using the quadrant symmetry (see Hypothesis~\ref{hypothesis:2}); an example of such a construction is given in Example~\ref{example:1}-(iv).
For all $\bfbeta\in\curB$, we define the function $[y]\mapsto \widetilde\chi_\bfbeta^\Delta ([y]): \curC_y\rightarrow \RR^+$ such that
\begin{equation} \label{eq:eq34ter}
\widetilde\chi_\bfbeta^\Delta ([y]) = a_\bfbeta([y])^2 (\sum_{\bfbeta'\in\curB} a_{\bfbeta'}([y])^2)^{-1}\, .
\end{equation}
The dimensionless spectral measure $\widetilde\mu_G^\nu(d\bftau; [y])$ for $[y]$ given in $\curC_y$ is then defined by
\begin{equation} \label{eq:eq34quar}
\widetilde\mu_G^\nu(d\bftau; [y]) = \sum_{\bfbeta\in\curB} \widetilde\chi_\bfbeta^\Delta ([y])\, \delta_{\bftau_\bfbeta}(\bftau) \, .
\end{equation}
Note that $[y]$ is one parameter of $\bfcurS$ and $[\underline y]$ is the corresponding parameter of $\underline\bfcurS$.
\end{definition}
%
%---- PROPOSITION 2 ----------------------
\begin{proposition}[Random discretized dimensionless spectral measure] \label{proposition:2}
Let us consider Definitions~\ref{definition:3} and \ref{definition:4}. It is assumed that Eq.~\eqref{eq:eq34plus} holds.

\noindent (i) $\forall\bfbeta\in\curB$, function $[y]\mapsto\widetilde\chi_\bfbeta^\Delta ([y])$ is continuous on $\curC_y$ (and thus bounded on $\curC_y$), is such that $\widetilde\chi_\bfbeta^\Delta ([\underline  y]) \simeq \chi_\bfbeta^\Delta$, and
$\forall\, [y]\in\curC_y$, $\sum_{\bfbeta\in\curB}\, \widetilde\chi_\bfbeta^\Delta ([y]) = 1$.

\noindent (ii) Let $[\bfY]$ be the $\MM_{3,\widehat\nu_s}$-valued random variable, defined on $(\Theta,\curT,\curP)$, whose support of its probability measure is $\curC_y\subset\MM_{3,\widehat\nu_s}$, and such that
$\{ [\bfY]_{j\widehat\beta} \, , \, j\in\{1,2,3\}\, , \, \widehat\beta \in\{ 1,\ldots,\widehat\nu_s\}\}$ are $3\,\widehat\nu_s$ independent uniform random variables on $[0,1]$. Its mean value is $E\{[\bfY]\} = \int_{\curC_y} [y]\, P_{[\bfY]}(dy) = \int_{\curC_y} [y]\, dy = [\underline y]$. For all $\widehat\bfbeta\in\widehat\curB$,
$A_{\widehat\bfbeta} = a_{\widehat\bfbeta}([\bfY])$ is a  second-order real-valued random variable.

\noindent (iii) $\forall\bfbeta\in\curB$, $\widetilde\chi_\bfbeta^\Delta ([\bfY])$ is a second-order positive-valued random variable, defined on $(\Theta,\curT,\curP)$ such that $\sum_{\bfbeta\in\curB} \widetilde\chi_\bfbeta^\Delta ([\bfY]) = 1$ almost surely.

\noindent (iv) The  dimensionless spectral measure $\widetilde\mu_G^\nu(d\bftau;[y])$ for given $[y]$ in $\curC_y$, is a bounded positive measure on $\RR^3$
and is such that $\widetilde\mu_G^\nu(d\bftau;[\underline y]) \simeq \mu_G^\nu(d\bftau)$.  For all $[y]\in\curC_y$,
$\widetilde\mu_G^\nu(\RR^3;[\underline y]) = \sum_{\bfbeta\in\curB}\, \widetilde\chi_\bfbeta^\Delta ([\underline  y]) = 1$.
\end{proposition}
%
%------ PROOF OF PROPOSITION 2 ------------------
\begin{proof} (Proposition~\ref{proposition:2}).
This proposition is easy to prove and is left to the reader.
\end{proof}
%
%---- EXAMPLE 1 ----------------------
\begin{example}[{\textit{Illustration of a construction for a separable spatial correlation structure}}] \label{example:1}
\noindent{\textit{(i) Spectral density function}}. $\forall\bfk=(k_1,k_2,k_3)\in \RR^3$,
$s(\bfk)=\Pi_{j=1}^3 s_j(k_j)$. For $j=1,2,3$, $s_j(k_j) = K_j^{-1}\, (1-\vert k_j\vert/K_j)\,\11_{[-K_j,K_j]}(k_j)$ and thus, supp $s_j =[-K_j,K_j]$,
$s_j(-k_j)=s_j(k_j)$ (yielding $s(-\bfk)=s(\bfk)$ and the quadrant symmetry), and $\int_{[-K_j,K_j]} s_j(k_j)\, dk_j = 1$.

\noindent {\textit{(ii) Correlation function and spatial correlation length}}. For all $\bfzeta =(\zeta_1,\zeta_2,\zeta_3)\in\RR^3$, $\rho_G(\bfzeta)=\Pi_{j=1}^3 \rho_j(\zeta_j)$ and for $j=1,2,3$, $\rho_j(\zeta_j) = \int_\RR e^{\iota\,k_j\zeta_j} \, s_j(k_j)\, dk_j$, $\rho_j(0)=1$, and the spatial correlation length is $L_{cj} = \int_0^{+\infty} \vert\rho_j(\zeta_j)\vert\, d\zeta_j = \pi\, s_j(0) = \pi/K_j$.

\noindent {\textit{(iii) Dimensionless spectral density function and spectral sampling}}. $\forall\bftau = (\tau_1,\tau_2,\tau_3)$, $\chi(\tau) = \Pi_{j=1}^3 \, \chi_j(\tau_j)$. For $ j=1,2,3$, $\chi_j(\tau_j) = (1-\vert\tau_j\vert)\, \11_{[-1,1]}(\tau_j)$ and therefore, supp $\chi_j = [-1,1]$, $\chi_j(-\tau_j) = \chi_j(\tau_j)$, and
$\int_\RR \chi_j(\tau_j)\, d\tau_j =1$. For all $\bfbeta=(\beta_1,\beta_2,\beta_3)\in\curB$, $\chi_\bfbeta^\Delta =\Pi_{j=1}^3 \chi_{j\beta_j}^\Delta$
with $\chi_{j\beta_j}^\Delta = (2/\nu_s) \, \chi_j(\tau_{\beta_j})$.

\noindent {\textit{(iv) Construction of}} $a_\bfbeta([y])$. For $j\in\{1,2,3\}$, $\widehat\beta_j \in\{1,\ldots , \widehat\nu_s\}$,
$\widehat\bfbeta = (\widehat\beta_1,\widehat\beta_2,\widehat\beta_3)$, and $\forall [y]\in\curC_y$,
$q_{\widehat\bfbeta}([y];\delta_s) = \Pi_{j=1}^3 q_{j\widehat\beta_j}([y];\delta_j)$ in which $q_{j\widehat\beta_j}([y];\delta_j)= 1+\sqrt{12}\, \delta_j\,([y]_{j\widehat\beta_j} -1/2)$ with $\delta_j > 0$ the hyperparameter. We thus have $a_{\widehat\bfbeta}([y]) = \Pi_{j=1}^3 a_{j\widehat\beta_j}([y])$ in which
$a_{j\widehat\beta_j}([y])=\sqrt{\chi^\Delta_{j\widehat\beta_j}}\, q_{j\widehat\beta_j}([y],\delta_j)$ and for $\beta_j\in\{\widehat\nu_s \! + \! 1,\ldots , 2\,\widehat\nu_s\}$, $a_{j\beta_j}([y])= a_{j, (2\,\widehat\nu_s+1-\beta_j)}([y])$.

\noindent {\textit{(v) Random variable}} $A_{\widehat\bfbeta}$ {\textit{and hyperparameter}} $\delta_s$. For $j\in\{1,2,3\}$ and $\widehat\beta_j \in\{1,\ldots , \widehat\nu_s\}$, the mean value and the second-order moment of random variable $A_{j\widehat\beta_j} = a_{j\widehat\beta_j}([\bfY])$ are $E\{A_{j\widehat\beta_j}\} = \sqrt{\chi^\Delta_{j\widehat\beta_j}}$ and $E\{A_{j\widehat\beta_j}^2\} = \chi^\Delta_{j\widehat\beta_j}\, (1+\delta_j^2)$. Since the random variables
$\{A_{j\widehat\beta_j}\}_{j,\widehat\beta_j}$ are independent, the mean value and the second-order moment of the random variable
$A_{\widehat\bfbeta} = a_{\widehat\bfbeta}([\bfY]) = \Pi_{j=1}^3 A_{j\widehat\beta_j}$ are
$E\{A_{\widehat\bfbeta}\} = \sqrt{\chi^\Delta_{\widehat\bfbeta}}$ and $E\{A_{\widehat\bfbeta}^2\} = \chi^\Delta_{\widehat\bfbeta}\, \Pi_{j=1}^3(1+\delta_j^2)$.
Defining the hyperparameter $\delta_s$ as $\delta_s^2= E\{(A_{\widehat\bfbeta}-\sqrt{\chi^\Delta_{\widehat\bfbeta}})^2\}/ \chi^\Delta_{\widehat\bfbeta}$, it can be seen that we have $\delta_s^2= (\Pi_{j=1}^3(1+\delta_j^2)) - 1 > 0$, which is independent of $\widehat\bfbeta$.

\noindent {\textit{(vi) Discretized dimensionless spectral measure}}.  Eq.~\eqref{eq:eq34ter} yields, $\forall\bfbeta=(\beta_1,\beta_2,\beta_3)\in\curB$,
$\widetilde\chi_\bfbeta^\Delta ([y]) = \Pi_{j=1}^3 \,\widetilde\chi_{j\beta_j}^\Delta ([y])$
in which $\widetilde\chi_{j\beta_j}^\Delta ([y]) = a_{j\beta_j}([y])^2 (\sum_{\bfbeta'\in\curB} a_{\bfbeta'}([y])^2)^{-1/3}$.
\end{example}
%
%
%---DEFINITION 5 ----------------
\begin{definition}[{\textit{Spectrum parameters}} $\bfw$ {\textit{and}} $\bfcurS$ {\textit{and their probabilistic models}} $\bfW$ {\textit{and}} $\bfS$] \label{definition:5}
Let $\bfw=(w_1,w_2,w_3)$ in which $w_j=\pi/K_j$ (this parameter allows the support of the spectral measure to be controlled).
Let $\curC_w =\{\bfw\in\RR^3; w_j\in [w_j^\pmin,w_j^\pmax]$ for $j=1,2,3\}$ be the compact subset of $\RR^3$, in which $0< w_j^\pmin = \pi/K_j^\pmax < w_j^\pmax = \pi/K_j^\pmin <+\infty$ (see Hypothesis~\ref{hypothesis:2}).
Parameter $\bfw$ is modeled by a $\RR^3$-valued random variable $\bfW$, defined on $(\Theta,\curT,\curP)$, independent of $[\bfY]$, whose support of its given probability measure $P_\bfW(d\bfw)$ is $\curC_w$.
We define the parameter $\bfcurS$ as $\{\bfw, [y]\}$, which  takes its values in the subset $\curC_\curS = \curC_w\times \curC_y$ of $\RR^3\times \MM_{3,\widehat\nu_s}$. The probabilistic model of $\bfcurS$ is the  $\RR^3\times \MM_{3,\widehat\nu_s}$- valued random variable $\bfS=\{\bfW,[\bfY]\}$ whose probability measure is the product of measures $P_\bfS  = P_\bfW(d\bfw)\otimes P_{[\bfY]}(dy)$ whose compact support is $\curC_\curS$.
\end{definition}
%
%---DEFINITION 6 ----------------
\begin{definition} [{\textit{Normalized Gaussian random field}} $G^\nu(\cdot;\bfcurS)$ {\textit{given}} $\bfcurS = \{ \bfw, \hbox{[} y \hbox{]}\}\in \curC_\curS = $ $ \curC_w\times \curC_y$ ] \label{definition:6}
Let $\nu=(\nu_s)^3$ be fixed.
Let $\{Z_\bfbeta, \bfbeta\in\curB\}$ and $\{\Phi_\bfbeta, \bfbeta\in\curB\}$ be $2\, \nu$ independent random variables on $(\Theta,\curT,\curP)$, which are independent of $\bfW$ and $[\bfY]$. For all $\bfbeta\in\curB$,  $Z_\bfbeta = \sqrt{-log \Psi_\bfbeta}$  in which $\Psi_\bfbeta$ is uniform of $[0,1]$ and $\Phi_\bfbeta$ is uniform on $[0, 2\,\pi]$. Let $P_\bfZ(d\bfz)$ and $P_\bfPhi(d\bfvarphi)$ be the probability measures on $\RR^\nu$ of the $\RR^\nu$-valued random variables $\bfZ=\{Z_\bfbeta,\bfbeta\in\curB\}$ and $\bfPhi =\{\Phi_\bfbeta,\bfbeta\in\curB\}$. The unbounded support $\curC_z$ of $P_\bfZ(d\bfz)$ is
$\curC_z = \{\bfz=\{z_\bfbeta,\bfbeta\in\curB\}, z_\bfbeta > 0\}\subset\RR^\nu$ and the compact support $\curC_\varphi$ of $P_\bfPhi(d\bfvarphi)$ is
$\curC_\varphi = \{\bfvarphi=\{\varphi_\bfbeta,\bfbeta\in\curB\}, \varphi_\bfbeta \in [0, 2\,\pi]\}\subset\RR^\nu$.
Let $\bfx\mapsto g^\nu(\bfx;\bfcurS,\bfz,\bfvarphi) : \RR^3\rightarrow \RR$ be such that, for all $\{\bfcurS,\bfz,\bfvarphi\}\in \curC_\curS\times \curC_z\times \curC_\varphi$,
\begin{equation} \label{eq:eq35}
g^\nu(\bfx;\bfcurS,\bfz,\bfvarphi) = \sum_{\bfbeta\in\curB} \sqrt{2\widetilde\chi_\bfbeta^\Delta ([y])}\, z_\bfbeta\, \cos (\varphi_\bfbeta +\sum_{j=1}^3 \frac{\pi}{w_j}\tau_{\beta_j} x_j) \, .
\end{equation}
For all $\bfw\in\curC_w$ and $[y]\in\curC_y$, we define the real-valued random field $\{G^\nu(\bfx;\bfcurS), \bfx\in\RR^3\}$ with
$\bfcurS = \{ \bfw, [y]\}\in \curC_\curS = \curC_w\times \curC_y$, such that
\begin{equation} \label{eq:eq36}
G^\nu(\bfx;\bfcurS) = g^\nu(\bfx;\bfcurS,\bfZ,\bfPhi) \, .
\end{equation}
Equation~\eqref{eq:eq36} with Eq.~\eqref{eq:eq35} corresponds to a finite discretization of the stochastic integral representation with a stochastic spectral measure for homogeneous second-order mean-square continuous random fields \cite{Doob1953,Guikhman1980,Kree1986}.
\end{definition}
%
%---- PROPOSITION 3 ----------------------
\begin{proposition}[Properties of random field $G^\nu(\cdot;\bfcurS)$] \label{proposition:3}
For all $\bfcurS = \{ \bfw, [y]\}\in \curC_\curS = \curC_w\times \curC_y$, the real-valued random field $\{ G^\nu(\bfx;\bfcurS),\bfx\in\RR^3\}$ is Gaussian, homogeneous, second-order, mean-square continuous, and normalized,
\begin{equation} \label{eq:eq37}
E\{G^\nu(\bfx;\bfcurS)\} = 0 \quad , \quad  E\{G^\nu(\bfx;\bfcurS)^2\} = 1 \quad , \quad \forall\bfx\in\RR^3 \, .
\end{equation}
Its dimensionless spectral measure $\widetilde\mu_G^\nu(d\bftau;[y])$, expressed with the dimensionless spectral variable $\bftau = (\tau_1,\tau_2,\tau_3)$ with $\tau_j=k_j/K_j$, is the spectral measure defined by Eq.~\eqref{eq:eq34quar}.
\end{proposition}
%
%------ PROOF OF PROPOSITION 3 ------------------
\begin{proof} (Proposition~\ref{proposition:3}).
Since $\{Z_\bfbeta, \bfbeta\in\curB\}$ and $\{\Phi_\bfbeta, \bfbeta\in\curB\}$  are $2\,\nu$ independent random variables, it can easily be proven
that $G^\nu(\cdot;\bfcurS)$ is centered (first equation in Eq.~\eqref{eq:eq37}) and, $\forall\bfzeta=(\zeta_1,\zeta_2,\zeta_3)\in\RR^3$ and $\forall\bfcurS = \{ \bfw, [y]\}\in \curC_\curS = \curC_w\times \curC_y$,
\begin{equation} \label{eq:eq39}
\rho_G^\nu(\bfzeta;\bfcurS) =  E\{G^\nu(\bfx\!+\!\bfzeta;\bfcurS)\, G^\nu(\bfx;\bfcurS)\} =  \sum_{\bfbeta\in\curB} \widetilde\chi_\bfbeta^\Delta ([y])
   \, \cos (\sum_{j=1}^3 \frac{\pi}{w_j}\tau_{\beta_j} \zeta_j)  \, .
\end{equation}
Using Proposition~\ref{proposition:2}-(i)  yields the second equation in Eq.~\eqref{eq:eq37}.
For all $\bfbeta\in\curB$, the random variable $Z_\bfbeta\, \cos (\Phi_\bfbeta +\sum_{j=1}^3 \frac{\pi}{w_j}\tau_{\beta_j} x_j)$ is Gaussian and consequently, random field $G^\nu(\cdot;\bfcurS)$ is Gaussian.
Since $G^\nu(\cdot;\bfcurS)$ is a Gaussian random field with zero mean function and a correlation function that depends only on $\bfzeta$, $G^\nu(\cdot;\bfcurS)$ is homogeneous on $\RR^3$.
Since $\bfzeta\mapsto \rho_G^\nu(\bfzeta;\bfcurS)$ defined by Eq.~\eqref{eq:eq39} is continuous on $\RR^3$, $G^\nu(\cdot;\bfcurS)$ is mean-square continuous on $\RR^3$ and thus there exists a spectral measure given by Eq.~\eqref{eq:eq34quar}. Note that the spectral measure in $\bfk=(k_1,k_2,k_3)$ is such that
$\widetilde m_G^\nu(d\bfk;\bfcurS) = \widetilde\mu_G^\nu(d\bftau;[y])$ with
$\widetilde m_G^\nu(d\bfk;\bfcurS) = \sum_{\bfbeta\in\curB} \widetilde s_\bfbeta^\Delta(\bfw,[y])\, \delta_{\bfk_\bfbeta}(\bfk)$
in which $\widetilde s_\bfbeta^\Delta(\bfw,[y]) =\widetilde\chi_\bfbeta^\Delta([y])\, \Pi_{j=1}^3 (\pi/w_j)$ with $\bfcurS = \{\bfw,[y]\}$.
\end{proof}

\section{Non-Gaussian random field $\hbox{[} \CC(.;\bfcurS) \hbox{]}$ parameterized by $\bfcurS$ and random field $\hbox{[} \widetilde \CC \hbox{]}$ with uncertain spectral measure}
\label{sec:Section4}
In the construction of $G$ with an uncertain spectrum, the spectral measure $m_G(d\bfk)$ of $G$  is given (see Hypothesis~~\ref{hypothesis:2}). This is the reason why the convergence of the sequence $\{m_G^\nu(d\bfk)\}_\nu$ of measures towards $m_G(d\bfk)$ has been studied (see Lemma~\ref{lemma:4}).  The uncertain dimensionless spectrum, represented by $\widetilde\mu_G^\nu(d\bftau;[y])$ for $[y]$ given in $\curC_y$, is constructed from $\mu_G^\nu(d\bftau) = m_G^\nu(d\bfk)$ and constitutes the uncertain spectral measure of random field $G^\nu(\cdot; \bfcurS)$ given $\bfcurS=\{\bfw,[y]\}\in\curC_\curS=\curC_w\times\curC_y$.
Although a limit  $G^\infty(\cdot; \bfcurS)$ of random field $G^\nu(\cdot; \bfcurS)$ exits for $\nu\rightarrow +\infty$ (see \cite{Poirion1995}), a convergence analysis is not useful for the probabilistic construction that is proposed because the limit is not given (unknown). The value of $\nu = (\nu_s)^3$ (see Definition~\ref{definition:2}) is chosen sufficiently large in order that Hypothesis~\ref{hypothesis:3} be verified. Proposition~\ref{proposition:3} shows that, for all $\bfcurS\in\curC_\curS$, the random field
$G^\nu(\cdot; \bfcurS)$, defined by Eq.~\eqref{eq:eq36}, satisfies all the required properties (Gaussian, homogeneous, mean-square continuous, and normalization). We are therefore led to introduce the following definition in coherence with
Proposition~\ref{proposition:1}, Lemma~\ref{lemma:2}, and Remark~\ref{remark:1}.
%
%---DEFINITION 7 ----------------
\begin{definition} [{\textit{Random field}} $\hbox{[} \bfC(\cdot,\bfcurS)\hbox{]}$ {\textit{given}} $\bfcurS$] \label{definition:7}
We assume that $\nu$ is fixed and satisfies Hypothesis~\ref{hypothesis:3}.
The non-Gaussian random field $\bfC(\cdot,\bfcurS)$ given $\bfcurS\in\curC_\curS$ is defined by Eq.~\eqref{eq:eq11} in which the $21$ Gaussian random fields
$\{ G_{mn}(\bfx;\bfcurS),\bfx\in\RR^3\}_{1\leq m\leq n\leq 6}$ are replaced by $21$ independent copies of the Gaussian real-valued random field $\{G^\nu(\bfx;\bfcurS),\bfx\in\RR^3\}$ defined by Eq.~\eqref{eq:eq36}, and denoted by $\{G^\nu_{mn}(\bfx;\bfcurS),\bfx\in\RR^3\}_{1\leq m\leq n\leq 6}$.
For all $\bfcurS\in\curC_\curS$, $\bfx\in\RR^3$, and for $1\leq m\leq n\leq 6$, using Eq.~\eqref{eq:eq36} yields
\begin{equation} \label{eq:eq40}
G^\nu_{mn}(\bfx;\bfcurS) = g^\nu(\bfx;\bfcurS,\bfZ^{mn},\bfPhi^{mn}) \, .
\end{equation}
in which $\{\bfZ^{mn},\bfPhi^{mn}\}_{1\leq m\leq n\leq 6}$ are $21$ independent copies of $\RR^\nu$-valued random variables $\bfZ$ and $\bfPhi$
(see Definition~\ref{definition:6}), and we have
\begin{equation} \label{eq:eq41}
E\{G^\nu_{mn}(\bfx;\bfcurS)\} = 0 \quad , \quad  E\{G^\nu_{mn}(\bfx;\bfcurS)^2\} = 1 \, .
\end{equation}
\end{definition}
%
%
%---- PROPOSITION 4 ----------------------
\begin{proposition}[Properties of the non-Gaussian $\MM^+_6$-valued random field $ \{\hbox{[} \bfC(\cdot;\bfcurS) \hbox{]}\} $] \label{proposition:4}
The non-Gaussian random field $[\bfC(\cdot;\bfcurS)]$, defined in Definition~\ref{definition:7} for a given uncertain spectral measure parameterized by $\bfcurS$, is a second-order random field such that
\begin{equation} \label{eq:eq43}
\Vert\, [\bfC(\bfx;\bfcurS)]\,\Vert_F \,\,\, \leq \,\,\Gamma_C \,\,\, a.s. \quad ,\quad \forall\bfx\in\RR^3 \, ,
\end{equation}
in which $\Gamma_C$ is a second-order positive-valued random variable, independent of $\bfx$ and $\bfcurS$, such that
\begin{equation} \label{eq:eq44}
E\{\Gamma_C^2\} = \underline\gamma_{2,C}^2 < +\infty \quad , \quad  E\{\Gamma_C^4\} = \underline\gamma_{4,C}^4 < +\infty\, .
\end{equation}
For all $\bfomega$ and $\bfomega'$ in $\RR^6$ and for all $\bfx$ in $\RR^3$,
\begin{equation} \label{eq:eq45}
\vert \langle [\bfC(\bfx;\bfcurS)]\,\bfomega\, , \bfomega' \rangle_2 \vert  \,\,\, \leq \,\, \Gamma_C \, \Vert\bfomega\Vert_2\, \Vert\bfomega'\Vert_2 \,\,\, a.s.
\end{equation}
\end{proposition}
%
%------ PROOF OF PROPOSITION 4 ------------------
\begin{proof} (Proposition~\ref{proposition:4}).
For all $\bfx\in\RR^3$ and $\bfcurS\in\curC_\curS$, Proposition~\ref{proposition:1} and Definition~\ref{definition:7} yield
$[\bfC(\bfx;\bfcurS)] \! = \! [\bfL(\bfx;\bfcurS)]^T [\bfL(\bfx;\bfcurS)]$ in which
$[\bfL(\bfx;\bfcurS)]_{mn} \! = \! \sigma_c\, G_{mn}^\nu(\bfx;\bfcurS)$ for $1\!\leq \! m \! < \! n \! \leq \! 6$ and
$[\bfL(\bfx;\bfcurS)]_{mm} \! = \! \sigma_c\, \sqrt{ 2 h( G_{mm}^\nu(\bfx;\bfcurS);\alpha_m) }$ for $1\!\leq \! m\! = \! n\! \leq \! 6$.
We thus have, almost surely,
$\Vert\, [\bfC(\bfx;\bfcurS)]\,\Vert_F \,\, \leq \,\, \Vert [\bfL(\bfx;\bfcurS)] \Vert_F^2 =
\sum_m [\bfL(\bfx;\bfcurS)]_{mm}^2 + \sum_{m<n}[\bfL(\bfx;\bfcurS)]_{mn}^2
 = \sigma_c^2(\sum_m 2\, h( G_{mm}^\nu(\bfx;\bfcurS);$ $\alpha_m) +\sum_{m<n} G_{mn}^\nu(\bfx;\bfcurS)^2)$.
Using Lemma~\ref{lemma:2}-(i) yields
\begin{equation} \label{eq:eq46}
\Vert \, [\bfC(\bfx;\bfcurS)]\,\Vert_F \,\, \leq \,\, \sigma_c^2\,( 4\sum_m\alpha_m +2\sum_m G_{mm}^\nu(\bfx;\bfcurS)^2 + \sum_{m<n} G_{mn}^\nu(\bfx;\bfcurS)^2 )\, .
\end{equation}
For all $\bfcurS =\{\bfw,[y]\}\in\curC_\curS=\curC_w\times\curC_y$, $\bfz\in\curC_z$, and $\bfvarphi\in\curC_\varphi$, Eq.~\eqref{eq:eq35} allows for writing
$\vert g^\nu(\bfx;\bfcurS,\bfz,\bfvarphi)\vert \leq \sum_{\bfbeta\in\curB} \sqrt{ 2\widetilde\chi_\bfbeta^\Delta([y]) } \,$ $z_\bfbeta \leq
\sqrt{2 \sum_{\bfbeta\in\curB} \widetilde\chi_\bfbeta^\Delta([y])}\, \sqrt{\sum_{\bfbeta\in\curB} z_\bfbeta^2} \leq \sqrt{2 \sum_{\bfbeta\in\curB} z_\bfbeta^2}$
because, from Proposition~\ref{proposition:2}-(i), $\sum_{\bfbeta\in\curB}\widetilde\chi_\bfbeta^\Delta([y]) = 1$.
For all $\bfx\in\RR^3$, using Eq.~\eqref{eq:eq40} yields
$G_{mn}^\nu(\bfx;\bfcurS)^2 \leq 2 \sum_{\bfbeta\in\curB} (Z_\bfbeta^{mn})^2$ almost surely.
Using Eq.~\eqref{eq:eq46}, we obtain Eq.~\eqref{eq:eq43} in which $\Gamma_C = \sigma_c^2\,( 4\sum_m\alpha_m +4\sum_m \sum_\bfbeta (Z_\bfbeta^{mm})^2 + 2 \sum_{m<n} \sum_\bfbeta (Z_\bfbeta^{mn})^2 )$ that is independent of $\bfx$ and $\bfcurS$. The $21\times \nu$ random variables $\{Z_\bfbeta^{mn}\}_{mn,\bfbeta}$ are independent copies of random variable $Z = \sqrt{-log\Psi}$ whose probability measure is $P_Z(dz) = \11_{\RR^+}(z)\, 2z\, \exp(-z^2)\, dz$. We have $E\{Z\} = \sqrt{\pi}/2$ and $E\{Z^2\} =1$. Hence, $E\{\Gamma_C\} = \underline \gamma_{1,C}$ with $ \underline \gamma_{1,C}= \sigma_c^2(4\sum_m\alpha_m + 54\, \nu) < +\infty$. For $p=2$ or $4$,
$E\{\Gamma_C^p\} = 2 (2\sigma_c^2)^p\int_0^{+\infty}(2\sum_m\alpha_m + 27\nu z^2)^p\, z\, \exp(-z^2)\, dz = \underline\gamma_{p,C}^p < +\infty$ that yields Eq.~\eqref{eq:eq44} (note that $E\{\Gamma_C^4\} < +\infty$ implies $E\{\Gamma_C\} < +\infty$ and $E\{\Gamma_C^2\} < +\infty$).
Since $E\{ \Vert \, [\bfC(\bfx;\bfcurS)]\,\Vert_F^2\} \leq E\{\Gamma_C^2\} < +\infty$, $[\bfC(\cdot;\bfcurS)]$ is a second-order random field.
Finally, for all $\bfx\in\RR^3$, $\bfomega$ and $\bfomega'$ in $\RR^6$,
$\vert \langle [\bfC(\bfx;\bfcurS)]\,\bfomega\, , \bfomega' \rangle_2 \vert  \,\,
\leq \Vert [\bfC(\bfx;\bfcurS)] \Vert_2 \,\Vert\bfomega\Vert_2\, \Vert \bfomega' \Vert_2
\leq \Vert [\bfC(\bfx;\bfcurS)] \Vert_F \,\Vert\bfomega\Vert_2\, \Vert \bfomega' \Vert_2 $,
which yields Eq.~\eqref{eq:eq45} using  Eq.~\eqref{eq:eq43}.
\end{proof}

%---- COROLLARY 1 ----------------------
\begin{corollary}[Properties of the non-Gaussian $\MM^+_6$-valued random field $\hbox{[} \CC(\cdot;\bfcurS) \hbox{]}$] \label{corollary:1}
Let $\CC(\cdot;\bfcurS)$ be the non-Gaussian second-order random field defined by Eq.~\eqref{eq:eq2} in which $\bfC(\cdot;\bfcurS)$ satisfies the  properties given in Proposition~\ref{proposition:4}. For all $\bfomega\in\RR^6\backslash\{0\}$,
\begin{equation} \label{eq:eq47}
\frac{\langle [\CC(\bfx;\bfcurS)]\,\bfomega  , \bfomega \rangle_2}{\Vert\bfomega\Vert_2^2} \leq \Vert\, [\CC(\bfx;\bfcurS)]\,\Vert_F \,\,\, \leq \,\,\Gamma_\CC \,\,\, a.s. \quad ,\quad \forall\bfx\in\RR^3 \, ,
\end{equation}
with $\Gamma_\CC$ a second-order $\RR^+$-valued random variable, independent of $\bfx$ and $\bfcurS$, such that
\begin{equation} \label{eq:eq48}
E\{\Gamma_\CC^2\} = \underline\gamma_{2,\CC}^2 < +\infty \quad , \quad  E\{\Gamma_\CC^4\} = \underline\gamma_{4,\CC}^4 < +\infty\, .
\end{equation}
For all $\bfomega$ and $\bfomega'$ in $\RR^6$ and for all $\bfx$ in $\RR^3$,
\begin{align}
\vert \langle [\CC(\bfx;\bfcurS)]\,\bfomega , \bfomega'\rangle_2 \vert  \,\,\, & \leq \,\, \Gamma_\CC \, \Vert\bfomega\Vert_2\,
                                                                                            \Vert\bfomega'\Vert_2 \,\,\, a.s. \, , \label{eq:eq49} \\
\langle [\CC(\bfx;\bfcurS)]\,\bfomega , \bfomega \rangle_2  \,\,\, & \geq \,\, \underline c_\epsilon \, \Vert\bfomega\Vert_2^2 \,\,\, a.s. \label{eq:eq49bis}
\end{align}
in which $\underline c_\epsilon  = \underline c_0\,\epsilon/(1+\epsilon)$ is a finite positive constant independent of $\bfx$ and $\bfcurS$.
\end{corollary}
%
%
%------ PROOF OF COROLLARY 1 ------------------
\begin{proof} (Corollary~\ref{corollary:1}).
Equation~\eqref{eq:eq2} yields
$\Vert\, [\CC(\bfx;\bfcurS)]\,\Vert_F\, \leq (1+\epsilon)^{-1} \Vert [\underline\LL]^T\Vert_F\,\Vert [\underline\LL]\Vert_F$
          $(\epsilon\,\Vert[I_6]\Vert_F + \Vert [\bfC(\bfx;\bfcurS)\Vert_F)$. We have
$\Vert [\underline\LL]^T\Vert_F = \Vert [\underline\LL]\Vert_F = (\tr[\underline\CC])^{1/2}$ and $\Vert[I_6]\Vert_F =\sqrt{6}$.
Eqs.~\eqref{eq:eq1} and  \eqref{eq:eq43} yield  $\Vert\, [\CC(\bfx;\bfcurS)]\,\Vert_F\,\, \leq \, \underline c_1 (1+\epsilon)^{-1}(\epsilon\sqrt{6} + \Vert\, [\bfC(\bfx;\bfcurS)]\,\Vert_F ) \leq \, \Gamma_\CC$ almost surely with $\Gamma_\CC =\underline c_1 (1+\epsilon)^{-1}(\epsilon\sqrt{6} + \Gamma_C)$, which is the second part of Eq.~\eqref{eq:eq47}. For all $\bfomega$ and $\bfomega'$ in $\RR^6$, we have
$\vert \langle [\CC(\bfx;\bfcurS)]\,\bfomega , \bfomega' \rangle_2 \vert  \,\,  \leq \,\Vert [\CC(\bfx;\bfcurS)]\Vert_F  \, \Vert\bfomega\Vert_2\,\Vert\bfomega'\Vert_2$.
Taking $\bfomega'=\bfomega\in\RR^6\backslash \{0\}$ yields the first part of  Eq.~\eqref{eq:eq47}. Then using the second part of Eq.~\eqref{eq:eq47} yields Eq.~\eqref{eq:eq49}.
From Eq.~\eqref{eq:eq2}, it can be deduced that $\langle[\CC(\bfx;\bfcurS)]\,\bfomega\, , \bfomega \rangle_2 =
(\epsilon \langle [\underline\CC]\,\bfomega , \bfomega \rangle_2 + \langle [\bfC(\bfx;\bfcurS)]\,[\underline\LL]\, \bfomega , [\underline\LL]\bfomega\rangle_2)/(1+\epsilon)$.
From Eq.~\eqref{eq:eq1} and since $[\bfC(\bfx;\bfcurS)]$ is a $\MM^+_6$-valued random variable, we obtain Eq.~\eqref{eq:eq49bis}.
\end{proof}
%
%---DEFINITION 8 ----------------
\begin{definition} [{\textit{Random field}} $\hbox{[} \widetilde\CC\hbox{]}$ {\textit{with uncertain spectral measure}}] \label{definition:8}
We assume that $\nu$ is fixed and satisfies Hypothesis~\ref{hypothesis:3}.
The random field $\{ [\widetilde\CC(\bfx)]\in\RR^3\}$ with uncertain spectral measure is defined by
\begin{equation} \label{eq:eq42}
[\widetilde\CC(\bfx)] = [\CC(\bfx;\bfS)] \quad , \quad \forall\bfx\in\RR^3 \, ,
\end{equation}
in which $\bfS = \{\bfW,[\bfY]\}$ is the $\RR^3\times \MM_{3,\widehat\nu_s}$- valued random variable defined by Definition~\ref{definition:5}.
\end{definition}
%
%----- REMARK 2 ------------------------------
\begin{remark} \label{remark:2}
Random field $\{ [\widetilde\CC(\bfx)]\in\RR^3\}$ defined by Eq.~\eqref{eq:eq42} is the random field with uncertain spectral measure. As explained in
Remark~\ref{remark:1}, $E\{[\CC(\bfx;\bfcurS)]\} = [\underline\CC]$ for all $\bfx\in\RR^3$,  we have not $E\{[\widetilde\CC(\bfx)]\} = [\underline\CC]$, but we have the approximation $E\{[\widetilde\CC(\bfx)]\} \simeq [\underline\CC]$.
\end{remark}

\section{Stochastic elliptic boundary value problem for stochastic homogenization}
\label{sec:Section5}
We consider a heterogeneous complex elastic microstructure occupying domain $\Omega$, which by definition is a microstructure that cannot be described  in terms of its constituents  at the microscale. This is typically the case of live tissues. In such a case, the stochastic model of the apparent elasticity field can be constructed at the mesoscale that corresponds to the scale  of the spatial correlation length of the microstructure $\Omega$ as proposed in \cite{Soize2006,Soize2008,Soize2017b}. The stochastic homogenization from the mesoscale to the macroscale allows the effective elasticity tensor $\widetilde\CC^\peff$ to be constructed. The study of the statistical properties of $\widetilde\CC^\peff$  allows for analyzing the scale separation. The separation is obtained if the statistical fluctuations of $\widetilde\CC^\peff$ are sufficiently small and, in this case, $\Omega$ is a representative volume element (RVE) \cite{Kanit2003,Ostoja2006,Ostoja2007,Soize2008}. Such a separation occurs if the spatial correlation length at the mesoscale is sufficiently small with respect to the characteristic geometrical dimension of $\Omega$. If not, $\widetilde\CC^\peff$ exhibits significant statistical fluctuations and therefore, $\Omega$ is not a RVE.

The deterministic part of the formulation used in Sections~\ref{sec:Section5.1} to \ref{sec:Section5.3} to write the problem of homogenization on $\Omega$ is that proposed in \cite{Bornert2008} for homogeneous deformations on the boundary $\partial\Omega$.
We use the convention for summations over repeated Latin indices $j$, $p$, and $q$ taking values in $\{1, 2, 3\}$.
\subsection{Definition of the stochastic boundary value problem (BVP)}
\label{sec:Section5.1}
Let $\Omega$ be a bounded open subset of $\RR^3$ with a sufficiently regular boundary $\partial\Omega$.
Let $\bfcurS$ be fixed in $\curC_\curS$. For all $\ell$ and $r$ in $\{1,2,3\}$, we have to find the $\RR^3$-valued random field
$\{\widetilde\bfU^{\ell r}(\bfx) = (\widetilde U_1^{\ell r}(\bfx), \widetilde U_2^{\ell r}(\bfx), \widetilde U_3^{\ell r}(\bfx) ) ,\bfx\in \overline{\Omega}\}$,
defined on $(\Theta,\curT,\curP)$, indexed by $\overline\Omega$, such that almost surely,
\begin{align}
 -\frac{\partial}{\partial x_j} ( \CC_{ijpq}(\bfx;\bfcurS)\, \varepsilon_{pq}( \widetilde\bfU^{\ell r}(\bfx) ) = 0
                                                        & \quad , \quad  \forall\bfx\in \Omega\quad , \quad i=1,2,3\, ,\label{eq:eq50} \\
 \widetilde\bfU^{\ell r}(\bfx) = \widetilde\bfu_0^{\ell r}(\bfx) &\quad , \quad  \forall\bfx\in \partial\Omega \, ,   \label{eq:eq51}
\end{align}
in which $\varepsilon_{pq}(\bfu) = (\partial u_p /\partial x_q + \partial u_q /\partial x_p)/2$ for $\bfu=(u_1,u_2,u_3)$
and where for all $\bfx\in\partial\Omega$, $\widetilde\bfu_0^{\ell r}(\bfx) = (\widetilde u_{0,1}^{\ell r}(\bfx), \widetilde u_{0,2}^{\ell r}(\bfx), \widetilde u_{0,3}^{\ell r}(\bfx) )$ is defined by $\widetilde u_{0,j}^{\ell r}(\bfx)=( \delta_{j\ell} x_r + \delta_{jr} x_\ell)/2$ with $\delta_{j\ell}$ the Kronecker symbol. The fourth-order tensor-valued random field $\{\CC_{ijpq}(\cdot;\bfcurS)\}_{ijpq}$ is such that
$\CC_{ijpq} = \CC_{jipq} = \CC_{ijqp} = \CC_{pqij}$ for $i$, $j$, $p$, and $q$ in $\{1,2,3\}$ and is such that
$\CC_{ijpq}(\cdot;\bfcurS) =[\CC(\cdot,\bfcurS)]_{\textbf{i}\textbf{j}}$ in which $\textbf{i} =(i,j)$ with $1\leq i\leq j\leq 3$ and $\textbf{j} =(p,q)$ with $1\leq p\leq q\leq 3$ are indices with values in $\{1,\ldots , 6\}$, and where the $\MM^+_6$-valued random field $[\CC(\cdot;\bfcurS)]$ is the one constructed in Section~\ref{sec:Section4} and whose properties are given by Corollary~\ref{corollary:1}.
\subsection{Random effective tensor  from stochastic homogenization and its random eigenvalues}
\label{sec:Section5.2}
For $\bfcurS$ fixed in $\curC_\curS$, for $i$, $j$, $\ell$, and  $r$ in $\{1,2,3\}$ the component $\CC_{ij\ell r}^\peff (\bfcurS)$ of the random fourth-order effective tensor $\CC^\peff (\bfcurS)$  is defined by
\begin{equation} \label{eq:eq52}
\CC_{ij\ell r}^\peff (\bfcurS) = \frac{1}{\vert\Omega\vert} \int_\Omega \CC_{ijpq}(\bfx;\bfcurS)\, \varepsilon_{pq}( \widetilde\bfU^{\ell r}(\bfx) )\, d\bfx \, ,
\end{equation}
in which $\widetilde\bfU^{\ell r}$ is the $\RR^3$-valued random field that satisfies Eqs.~\eqref{eq:eq50} and \eqref{eq:eq51} and where $\vert\Omega\vert = \int_\Omega d\bfx$. The fourth-order effective tensor $\CC^\peff (\bfcurS)$ satisfies the symmetry and positive-definiteness properties \cite{Bornert2008}. We can thus define the effective $\MM^+_6$-valued random matrix $[\CC^\peff (\bfcurS)]$ associated with random tensor  $\CC^\peff (\bfcurS)$, which is such that
$[\CC^\peff (\bfcurS)]_{\textbf{i}\textbf{j}} = \CC_{ij\ell r}^\peff (\bfcurS)$ in which $\textbf{i}=(i,j)$ with $1\leq i\leq j \leq 3$ and  $\textbf{j}=(\ell,r)$ with $1\leq \ell\leq r \leq 3$.
\subsection{Transforming the nonhomogeneous Dirichlet BVP in a homogeneous Dirichlet BVP}
\label{sec:Section5.3}
For fixed $\ell$ and $r$, since $\bfx\mapsto\widetilde\bfu_0^{\ell r}(\bfx)$ is a linear function in $\bfx$, we can perform the following translation (without having to resort the trace theorem in Hilbert spaces),
\begin{equation} \label{eq:eq53}
\widetilde\bfU^{\ell r}(\bfx) = \bfU^{\ell r}(\bfx) + \widetilde\bfu_0^{\ell r}(\bfx) \quad , \quad \forall\bfx \in\overline\Omega \, .
\end{equation}
Since, $\forall\bfx \in\overline\Omega$,
$\varepsilon_{pq}( \widetilde\bfu_0^{\ell r}(\bfx) )  =(\delta_{p\ell}\delta_{qr} +\delta_{pr}\delta_{q\ell})/2$ and
$\CC_{ij\ell r}(\bfx;\bfcurS) = \CC_{ijr\ell}(\bfx;\bfcurS)$, for all $\ell$ and $r$ in $\{1,2,3\}$, the nonhomogeneous Dirichlet BVP defined by
Eqs.~\eqref{eq:eq50} and \eqref{eq:eq51} becomes the following homogeneous Dirichlet BVP for the
$\RR^3$-valued random field $\{\bfU^{\ell r}(\bfx) = (U_1^{\ell r}(\bfx), U_2^{\ell r}(\bfx), U_3^{\ell r}(\bfx) ) ,\bfx\in \overline{\Omega}\}$,
defined on $(\Theta,\curT,\curP)$, indexed by $\overline\Omega$, such that almost surely,
\begin{align}
 -\frac{\partial}{\partial x_j} ( \CC_{ijpq}(\bfx;\bfcurS)\, \varepsilon_{pq}( \bfU^{\ell r}(\bfx) ) & = f_i^{\ell r}(\bfx;\bfcurS)
                                                         \quad , \quad  \forall\bfx\in \Omega\quad , \quad i=1,2,3\, ,\label{eq:eq54} \\
 \bfU^{\ell r}(\bfx) & = \bfzero \quad , \quad  \forall\bfx\in \partial\Omega \, ,   \label{eq:eq55}
\end{align}
in which $f_i^{\ell r}(\bfx;\bfcurS) = \frac{\partial}{\partial x_j} (\CC_{ij\ell r}(\bfx;\bfcurS)$ that, from Proposition~\ref{proposition:1} and Eqs.~\eqref{eq:eq35},
\eqref{eq:eq36}, and \eqref{eq:eq40}, exists almost surely.
\subsection{Analysis of the stochastic homogeneous Dirichlet BVP}
\label{sec:Section5.4}
(i) {\textit{Definition of random vector $\bfXi$}}. Let $\bfXi = \{ \{ \bfZ^{mn},$ $1\leq m\leq n\leq 6 \} , \{ \bfPhi^{mn}, 1\leq m\leq n\leq 6 \} \}$ be the second-order random variable on $(\Theta,\curT,\curP)$ with values in $\RR^{n_\xi}$ with $n_\xi = 2\times 21 \times \nu$,  whose probability measure is
$P_\bfXi =(\otimes_{m,n} P_{\bfZ^{mn}})\otimes (\otimes_{m,n} P_{\bfPhi^{mn}})$ in which
$P_{\bfZ^{mn}} = P_\bfZ$ and $P_{\bfPhi^{mn}} = P_\bfPhi$ for all $1\leq m\leq n\leq 6$ (see Definitions~\ref{definition:6} and \ref{definition:7}).
Let $\curC_\xi \subset \RR^{n_\xi}$ be the support of $P_\bfXi$, which is known and can easily be written.
Consequently, we have
$E\{\Vert\bfXi\Vert_2^2\} = \int_{\RR^{n_\xi}} \Vert\bfxi\Vert_2^2 \, P_\bfXi(d\bfxi) = \int_{\curC_\xi} \Vert\bfxi\Vert_2^2 \, P_\bfXi(d\bfxi) < +\infty$.\\

\noindent (ii) {\textit{Definition of mappings $\bfxi\mapsto \cc(\cdot;\bfcurS,\bfxi)$ and $\bfxi\mapsto \bfu^{\ell r}(\cdot;\bfcurS,\bfxi)$}}.
 For $\bfcurS$ fixed in $\curC_\curS$, the fourth-order tensor-valued random field $\CC(\cdot;\bfcurS)$ is written as
$\CC_{ijpq}(\bfx;\bfcurS) = [\CC(\bfx;\bfcurS)]_{\textbf{i}\textbf{j}}$ in which $\textbf{i}=(i,j)$ with $i\leq j$ and $\textbf{j}=(p,q)$ with $p\leq q$.
Taking into account the construction presented in Sections~\ref{sec:Section2} to \ref{sec:Section4}, the random field $[\CC(\cdot;\bfcurS)]$ is defined by a $\MM^+_6$-valued measurable mapping $\bfxi\mapsto [\cc(\cdot;\bfcurS,\bfxi)]$ on $\curC_\xi$ such that
$[\CC(\cdot;\bfcurS)] = [\cc(\cdot;\bfcurS,\bfXi)]$. Therefore, the fourth-order random field $\CC(\cdot;\bfcurS)$ is defined by a  measurable mapping $\bfxi\mapsto \cc(\cdot;\bfcurS,\bfxi)$ on $\curC_\xi$ such that
\begin{equation} \label{eq:eq58}
\CC_{ijpq}(\cdot;\bfcurS) = \cc_{ijpq}(\cdot;\bfcurS,\bfXi) \quad , \quad i,j,p,q \in \{1,2,3\} \, .
\end{equation}
Similarly, random field $\bfU^{\ell r}$ involved in the stochastic BVP, defined by Eqs.~\eqref{eq:eq54} and \eqref{eq:eq55}, depends only on $\bfcurS$ and $\bfXi$, and is defined by a measurable mapping $\bfxi\mapsto \bfu^{\ell r}(\cdot;\bfcurS,\bfxi)$ on $\curC_\xi$ such that
$\bfU^{\ell r} = \bfu^{\ell r}(\cdot;\bfcurS,\bfXi)$ for $\ell$ and $r$ in $\{1,2,3\}$.\\

\noindent (iii) {\textit{Definition of Hilbert space $\HH$}}. Let $\HH = \{\bfv=(v_1,v_2,v_3); v_j\in H^1(\Omega)$ for $j=1,2,3;$ $\bfv = \bfzero$ on $\partial\Omega\}$ be the Hilbert space equipped with the inner product and the associated norm,
\begin{equation} \label{eq:eq60}
\langle\bfu ,\bfv\rangle_\HH = \int_\Omega \varepsilon_{pq}(\bfu(\bfx))\, \varepsilon_{pq}(\bfv(\bfx))\, d\bfx \quad , \quad \Vert\bfv\Vert_\HH = (\langle\bfv ,\bfv\rangle_\HH)^{1/2}\, .
\end{equation}
Note that $\Vert\bfv\Vert_\HH$ is a norm on $\HH$ due to the Korn inequality and because $\bfv=\bfzero$ on $\partial\Omega$ (see for instance \cite{Dautray2013}). Introducing the matrix $[\varepsilon]$ such that $[\varepsilon]_{pq} = \varepsilon_{pq}$, Eq.~\eqref{eq:eq60} can be rewritten as
\begin{equation} \label{eq:eq60bis}
\langle\bfu ,\bfv\rangle_\HH = \int_\Omega \langle [\varepsilon(\bfu(\bfx))] \, , [\varepsilon(\bfv(\bfx))]\rangle_F d\bfx \quad ,
\quad \Vert\bfv\Vert^2_\HH = \int_\Omega \Vert \,[\varepsilon(\bfv(\bfx))] \, \Vert_F^2\, d\bfx \, .
\end{equation}

\noindent (iv) {\textit{Definition of bilinear form $b(\cdot,\cdot;\bfcurS,\bfxi)$ and linear form $\curL^{\ell r}(\cdot;\bfcurS,\bfxi)$}}.
For $\bfcurS$ fixed in $\curC_\curS$, for all $\bfxi\in\curC_\xi$, for $\ell$ and $r$ in $\{1,2,3\}$, and using Eq.~\eqref{eq:eq58}, we define the bilinear form $(\bfu,\bfv)\mapsto b(\bfu,\bfv;\bfcurS,\bfxi): \HH\times\HH \rightarrow \RR$ such that
\begin{equation} \label{eq:eq61}
b(\bfu,\bfv;\bfcurS,\bfxi) = \int_\Omega \cc_{ijpq}(\bfx;\bfcurS,\bfxi)\, \varepsilon_{pq}(\bfu(\bfx))\, \varepsilon_{ij}(\bfv(\bfx))\, d\bfx \, ,
\end{equation}
and the linear form $\bfv\mapsto \curL^{\ell r}(\bfv;\bfcurS,\bfxi): \HH\rightarrow\RR$ such that
\begin{equation} \label{eq:eq62}
\curL^{\ell r}(\bfv;\bfcurS,\bfxi) = - \int_\Omega \cc_{\ell r i j}(\bfx;\bfcurS,\bfxi)\, \varepsilon_{ij}(\bfv(\bfx))\, d\bfx \, ,
\end{equation}
whose right-hand side member in Eq.~\eqref{eq:eq62} is the transformation of $\int_\Omega f_i^{\ell r}(\bfx;\bfcurS)\, v_i(\bfx)\, d\bfx$ for which we have used $\bfv=\bfzero$ on $\partial\Omega$ and the symmetry property $\cc_{ij\ell r} =\cc_{ji\ell r} =\cc_{\ell rij}$.\\

\noindent (v) {\textit{Type of stochastic solution sought}}.
From a computational point of view, a solution of the stochastic BVP defined by Eqs.~\eqref{eq:eq54} and \eqref{eq:eq55} will be constructed by using the Monte Carlo simulation method. Consequently, we only need to analyze the strong stochastic solution of the weak formulation of this stochastic BVP and the weak stochastic solution is not useful. We then limit Proposition~\ref{proposition:5} to the strong stochastic solution.
%
%---- PROPOSITION 5 ----------------------
\begin{proposition}[Weak formulation of the stochastic homogeneous Dirichlet BVP and its strong stochastic solution] \label{proposition:5}
(i) For $\bfcurS$ fixed in $\curC_\curS$ and for $1\leq \ell\leq r \leq 3$, the weak formulation of the stochastic BVP defined by
Eqs.~\eqref{eq:eq54} and \eqref{eq:eq55} is: for $P_\bfXi$-almost all $\bfxi$ in $\curC_\xi\subset\RR^{n_\xi}$, find $\bfu^{\ell r}(\cdot;\bfcurS,\bfxi)$
in $\HH$ such that
\begin{equation} \label{eq:eq63}
b(\bfu^{\ell r}(\cdot ;\bfcurS,\bfxi),\bfv;\bfcurS,\bfxi) = \curL^{\ell r}(\bfv;\bfcurS,\bfxi) \quad , \quad \forall\bfv\in\HH\, .
\end{equation}

\noindent (ii) For $1\leq \ell\leq r \leq 3$, there exists a unique solution  $\bfu^{\ell r}(\cdot ;\bfcurS,\bfxi)\in\HH$ (strong stochastic solution) such that Eq.~\eqref{eq:eq63} holds and $\bfu^{r \ell}(\cdot ;\bfcurS,\bfxi) = \bfu^{\ell r}(\cdot ;\bfcurS,\bfxi)$.

\noindent (iii) The associated stochastic solution $\bfU^{\ell r}(\cdot;\bfcurS) = \bfu^{\ell r}(\cdot ;\bfcurS,\bfXi)$ is of second-order,
\begin{equation} \label{eq:eq64}
E\{\Vert\bfU^{\ell r}(\cdot ;\bfcurS)\Vert_\HH^2\} = \underline\gamma_u^2 < +\infty\, .
\end{equation}
\end{proposition}
%
%------ PROOF OF PROPOSITION 5 ------------------
\begin{proof} (Proposition~\ref{proposition:5}).

(i) Using Eqs.~\eqref{eq:eq58} and \eqref{eq:eq60} to \eqref{eq:eq62}, it is easy to prove that Eq.~\eqref{eq:eq63} is the weak formulation of
Eqs.~\eqref{eq:eq54} and \eqref{eq:eq55}.

(ii) Let $[\cc^{\ell r}(\bfx;\bfcurS,\bfxi)]$ be the $(3\times 3)$ real matrix such that
$[\cc^{\ell r}(\bfx;\bfcurS,\bfxi)]_{ij} = \cc_{\ell rij}(\bfx ;\bfcurS,\bfxi)$.
Therefore Eq.~\eqref{eq:eq62} can be rewritten as
$\curL^{\ell r}(\bfv;\bfcurS,\bfxi) = - \int_\Omega \langle [\cc^{\ell r}(\bfx;\bfcurS,\bfxi)]\, , [\varepsilon(\bfv(\bfx))]\rangle_F\, d\bfx$ and using Eq.~\eqref{eq:eq60bis}
yields
$$\vert\curL^{\ell r}(\bfv;\bfcurS,\bfxi) \vert \leq (\int_\Omega \,\Vert \, [\cc^{\ell r}(\bfx;\bfcurS,\bfxi)]\, \Vert_F^2 \, d\bfx)^{1/2}\, \Vert\bfv\Vert_\HH\, .$$
For all $\ell$ and $r$ in $\{1,2,3\}$ and since $\cc_{\ell rij} = \cc_{\ell rji}$, we have
$$\Vert \, [\cc^{\ell r}(\bfx;\bfcurS,\bfxi)]\, \Vert_F^2 = \sum_{i,j}\cc_{\ell rij}(\bfx;\bfcurS,\bfxi)^2 \leq 2\,\sum_{\ell'\leq r',i\leq j}
\cc_{\ell' r' ij}(\bfx;\bfcurS,\bfxi)^2 = 2\, \Vert \, [\cc(\bfx;\bfcurS,\bfxi)]\, \Vert_F^2$$
 in which we have used the notation
$[\cc(\bfx;\bfcurS,\bfxi)]_{\textbf{j}'\textbf{i}} = \cc_{\ell'r'ij}(\bfx;\bfcurS,\bfxi)$ in which $\textbf{j}'=(\ell',r')$
with $\ell'\leq r'$ and $\textbf{i}=(i,j)$ with $i\leq j$. Taking into account Proposition~\ref{proposition:4} and its proof,
Eqs.~\eqref{eq:eq47} and \eqref{eq:eq48} of Corollary~\ref{corollary:1} with its proof, show that $\Gamma_\CC=\gamma_\CC(\bfXi)$ in which $\bfxi\mapsto\gamma_\CC(\bfxi)$ is a positive-valued measurable mapping on $\curC_\xi$, which is independent of $\bfcurS$ and such that
\begin{equation} \label{eq:eq65}
E\{\Gamma_\CC^2\} = \int_{\curC_\xi} \gamma_\CC(\bfxi)^2\, P_\bfXi(d\bfxi) = \underline\gamma_\CC^2 < +\infty \, .
\end{equation}
Equation~\eqref{eq:eq47} shows that
\begin{equation} \label{eq:eq66}
\Vert \, [\cc(\bfx;\bfcurS,\bfxi)]\, \Vert_F \,\,\leq \,\gamma_\CC(\bfxi)\,\, , \,\,\hbox{for} \,\, P_\bfXi\hbox{ - almost all} \,\,\bfxi\,\,\hbox{in} \,\,\curC_\xi\, .
\end{equation}
It can then be deduced that
\begin{equation} \label{eq:eq67}
\vert\curL^{\ell r}(\bfv;\bfcurS,\bfxi) \vert \,\,\leq \, \sqrt{2} \,\vert\Omega\vert^{1/2} \, \gamma_\CC(\bfxi)\, \Vert\bfv\Vert_\HH \, ,
\end{equation}
which shows that linear form $\bfv \mapsto \curL^{\ell r}(\bfv;\bfcurS,\bfxi)$ is continuous on $\HH$ for $P_\bfXi$-almost $\bfxi$ in $\curC_\xi$.
Equation~\eqref{eq:eq49} shows that, $\forall\bfomega$ and $\bfomega'$ in $\RR^6$,
$\vert \langle [\cc(\bfx;\bfcurS,\bfxi)]\,\bfomega , \bfomega'\rangle_2 \vert \,\leq  \gamma_\CC(\bfxi) \, \Vert\bfomega\Vert_2\, \Vert\bfomega'\Vert_2$
for $P_\bfXi$-almost $\bfxi$ in $\curC_\xi$. Using Eq.~\eqref{eq:eq61} and taking into account the symmetry properties of $\varepsilon_{pq}$ and $\cc_{ijpq}$ yield, for all $\bfu$ and $\bfv$ in $\HH$,
$\vert b(\bfu,\bfv;\bfcurS,\bfxi)\vert \,\,\leq \, 2\, \gamma_\CC(\bfxi)\, \Vert\bfu\Vert_\HH \, \Vert\bfv\Vert_\HH$,
which shows that bilinear form $(\bfu,\bfv)\mapsto b(\bfu,\bfv;\bfcurS,\bfxi)$ is continuous on $\HH\times \HH$ for $P_\bfXi$-almost $\bfxi$ in $\curC_\xi$. Equation~\eqref{eq:eq49bis} shows that,
$\forall\bfomega\in\RR^6$,
$\langle[\cc(\bfx;\bfcurS,\bfxi)]\,\bfomega , \bfomega\rangle_2  \,\,\,\geq \,\, \underline c_\epsilon \, \Vert\bfomega\Vert_2^2$
for $P_\bfXi$-almost $\bfxi$ in $\curC_\xi$. From Eq.~\eqref{eq:eq61}, it can be deduced that, $\forall\bfv\in\HH$,
\begin{equation} \label{eq:eq69}
b(\bfv,\bfv;\bfcurS,\bfxi) \,\,\geq \, \underline c_\epsilon\, \Vert\bfv\Vert_\HH^2  \, ,
\end{equation}
which proves that bilinear form $(\bfu,\bfv)\mapsto b(\bfu,\bfv;\bfcurS,\bfxi)$ is coercive for $P_\bfXi$-almost $\bfxi$ in $\curC_\xi$.
Due to the  continuity and coercivity  of bilinear form $b(\cdot,\cdot;\bfcurS,\bfxi)$ and due to the continuity of linear form
$\curL^{\ell r}(\cdot;\bfcurS,\bfxi)$ for $P_\bfXi$-almost $\bfxi$ in $\curC_\xi$, the use of the Lax-Milgram theorem \cite{Lax1954,Lions2012} allows
for proving (ii) of the Proposition.

(iii) Taking  $\bfv = \bfu^{\ell r}(\cdot; \bfcurS,\bfxi)$ in Eq.~\eqref{eq:eq63} yields
$b(\bfu^{\ell r}(\cdot ;\bfcurS,\bfxi),\bfu^{\ell r}(\cdot ;\bfcurS,\bfxi);\bfcurS,\bfxi)$ $ = \vert\curL^{\ell r}(\bfu^{\ell r}(\cdot ;\bfcurS,\bfxi);\bfcurS,\bfxi)\vert$.
From Eq.~\eqref{eq:eq67}, it can be deduced that
$\vert\curL^{\ell r}(\bfu^{\ell r}(\cdot ;\bfcurS,\bfxi);\bfcurS,\bfxi)\vert
\,\,\leq$ $\sqrt{2} \,\vert\Omega\vert^{1/2} \, \gamma_\CC(\bfxi)\, \Vert\bfu^{\ell r}(\cdot ;\bfcurS,\bfxi)\Vert_\HH$ and using  Eq.~\eqref{eq:eq69} yield
$\underline c_\epsilon\, \Vert\bfu^{\ell r}(\cdot ;\bfcurS,\bfxi)\Vert_\HH^2 \leq b(\bfu^{\ell r}(\cdot ;\bfcurS,\bfxi),$ $\bfu^{\ell r}(\cdot ;\bfcurS,\bfxi);\bfcurS,\bfxi)$. Consequently, we obtain
\begin{equation} \label{eq:eq70}
\Vert\bfu^{\ell r}(\cdot ;\bfcurS,\bfxi)\Vert_\HH \,\,\,\leq  \frac{\sqrt{2} \,\vert\Omega\vert^{1/2}}{ \underline c_\epsilon} \, \gamma_\CC(\bfxi)  \, .
\end{equation}
Finally,
$$E\{\Vert\bfU^{\ell r}(\cdot ;\bfcurS)\Vert_\HH^2\} = \int_{\curC_\xi} \Vert\bfu^{\ell r}(\cdot ;\bfcurS,\bfxi)\Vert_\HH^2\, P_\bfXi(d\bfxi)
\leq {2 \,\vert\Omega\vert} \,\underline c_\epsilon^{-2} \int_{\curC_\xi} \gamma_\CC(\bfxi)^2\,  P_\bfXi(d\bfxi)$$
 and
using Eq.~\eqref{eq:eq65} yield
$E\{\Vert\bfU^{\ell r}(\cdot ;\bfcurS)\Vert_\HH^2\}\! =\! 2 \vert\Omega\vert  \underline\gamma_\CC^2 / \underline c_\epsilon^2$ that is Eq.~\eqref{eq:eq64}
with $\underline\gamma_u^2 = 2 \,\vert\Omega\vert \underline\gamma_\CC^2 /\underline c_\epsilon^2$.
\end{proof}
\subsection{Random eigenvalues of the random effective elasticity matrix}
\label{sec:Section5.5}
For $\bfcurS=\{\bfw,[y]\}\in\curC_\curS=\curC_w\times\curC_y$, the random effective elasticity matrix $[\CC^\peff(\bfcurS)]$ defined in Section~\ref{sec:Section5.2} can be written as $[\CC^\peff(\bfcurS)] = [\cc^\peff(\bfcurS,\bfXi)]$ in which $\bfxi\mapsto [\cc^\peff(\bfcurS,\bfxi)]$ is a $\MM^+_6$-valued measurable mapping on $\curC_\xi\subset\RR^{n_\xi}$. For all $\bfcurS\in\curC_\curS$ and $\bfxi\in\curC_\xi$, let $\lambda_1(\bfcurS,\bfxi) \geq \ldots \geq \lambda_6(\bfcurS,\bfxi) > 0$ be the eigenvalues of matrix $[\cc^\peff(\bfcurS,\bfxi)]\in \MM_6^+$ and let $\bflambda(\bfcurS,\bfxi) = (\lambda_1(\bfcurS,\bfxi), \ldots , \lambda_6(\bfcurS,\bfxi))\in (\RR^{+*})^6$.
Let $\bfS =\{\bfW,[\bfY]\}$ be the $\RR^3\times \MM_{3,\widehat\nu_s}$-valued random variable defined in Definition~\ref{definition:5} and Proposition~\ref{proposition:2}-(ii), for which the support of its probability measure $P_\bfS = P_\bfW(d\bfw)\otimes P_{[\bfY]}(dy)$ is subset
$\curC_\bfcurS =\curC_w\times\curC_y$ of $\RR^3\times \MM_{3,\widehat\nu_s}$.
The random effective elasticity matrix $[\widetilde\CC^\peff]$, which corresponds to the elasticity random field $\widetilde\CC$ for which its spectral measure is uncertain, can then be written as $[\widetilde\CC^\peff] = [\CC^\peff(\bfS)]=[\cc^\peff(\bfS,\bfXi)]$.
Let $\{\widetilde\Lambda_j = \lambda_j(\bfS,\bfXi) ,j=1,\ldots,6\}$ be the ordered (a.s) random eigenvalues of $[\widetilde\CC^\peff]$. Let $\widetilde\bfLambda=(\widetilde\Lambda_1,\ldots,\widetilde\Lambda_6)$ be  the $\RR^6$-valued random variable whose support of its probability measure  $P_{\widetilde\bfLambda}(d\widetilde\bflambda)$ is $(\RR^{+*})^6$.
The operator norm of $[\widetilde\CC^\peff]$ is $\Vert\,[\widetilde\CC^\peff] \,\Vert_2  = \widetilde\Lambda_1$.
%
%
%---- COROLLARY 2 ----------------------
\begin{corollary}[Second-order properties of the random eigenvalues] \label{corollary:2}
Under proposition~\ref{proposition:5}, $\widetilde\bfLambda$ is a second-order $\RR^6$-valued random variable,
\begin{equation}\label{eq:eq71}
E\{ \Vert\widetilde\bfLambda\Vert_2^2\} = \gamma_{\lambda}^2 < +\infty \, .
\end{equation}
\end{corollary}
%
%
%------ PROOF OF COROLLARY 2 ------------------
\begin{proof} (Corollary~\ref{corollary:2}).
From Eqs.~\eqref{eq:eq52}, \eqref{eq:eq53}, and \eqref{eq:eq62}, it can be deduced that
$$\cc_{ij\ell r}^\peff(\bfcurS,\bfxi) = -\vert\Omega\vert^{-1}\curL^{ij}(\bfu^{\ell r}(\cdot;\bfcurS,\bfxi);\bfcurS,\bfxi)
      +\vert\Omega\vert^{-1} \int_\Omega \cc_{ij\ell r}(\bfx;\bfcurS,\bfxi)\, d\bfx\, .$$
Since $(a+b)^2 \leq 2(a^2+b^2)$, we have
$$\Vert\bflambda(\bfcurS,\bfxi)\Vert_2^2 \,= \Vert\,[\cc^\peff(\bfcurS,\bfxi)]\,\Vert_F^2 = \sum_{i\leq j, \ell\leq r} \cc_{ij\ell r}^\peff(\bfcurS,\bfxi)^2 \leq 2 \vert\Omega\vert^{-2}\sum_{i\leq j, \ell\leq r} \left\{ \vert \curL^{ij}(\bfu^{\ell r}(\cdot;\bfcurS,\bfxi);\bfcurS,\bfxi)\vert^2 + ( \int_\Omega \cc_{ij\ell r}(\bfx;\bfcurS,\bfxi)\, d\bfx)^2\right\}\, .$$
Using Eq.~\eqref{eq:eq66} allows us to write
$\sum_{i\leq j, \ell\leq r}( \int_\Omega \cc_{ij\ell r}(\bfx;\bfcurS,\bfxi)\, d\bfx)^2 \leq \vert\Omega\vert\int_\Omega\Vert \,[\cc(\bfx;\bfcurS,\bfxi))]\,\Vert_F^2 \,d\bfx \leq \vert\Omega\vert^2 \, \gamma_\CC(\bfxi)^2$.
Using Eq.~\eqref{eq:eq67}  and \eqref{eq:eq70} yields
$\sum_{i\leq j, \ell\leq r} \vert \curL^{ij}(\bfu^{\ell r}(\cdot;\bfcurS,\bfxi);\bfcurS,\bfxi)\vert^2 \leq
84 \, \underline c_\epsilon^{-2}\,\vert\Omega\vert^2\,\gamma_\CC(\bfxi)^4$. It can then be deduced the inequality
$\Vert\bflambda(\bfcurS,\bfxi)\Vert_2^2 \,\,\leq 168\, \underline c_\epsilon^{-2}\, \gamma_\CC(\bfxi)^4 + 2\, \gamma_\CC(\bfxi)^2$ and consequently, we have
$E\{ \Vert\widetilde\bfLambda\Vert_2^2\} =\int_{\curC_\curS}  \int_{\curC_\xi} \Vert \bflambda(\bfcurS,\bfxi)\Vert_2^2\, P_\bfS(d\bfs)\otimes P_\bfXi(d\bfxi) \,\, \leq
168\, \underline c_\epsilon^{-2}\, E\{\Gamma_\CC^4\} + 2\,E\{\Gamma_\CC^2\}$. Using Eq.~\eqref{eq:eq48} yields  Eq.~\eqref{eq:eq71} with
$\gamma_\lambda^2 = 168\, \underline c_\epsilon^{-2}\, \underline\gamma_{4,\CC}^4 + 2\,\underline\gamma_{2,\CC}^2$.
\end{proof}
\subsection{Brief comments about numerical aspects of stochastic solver}
\label{sec:Section5.6}
The Monte Carlo simulation method \cite{Robert2005,Rubinstein2008} is used as stochastic solver. Let $\{(\bfcurS^\kappa,\bfxi^\kappa)\in \curC_\curS\times\curC_\xi,\, \kappa=1,\ldots ,\kappa_\psim\}$ be $\kappa_\psim$ independent realizations of random variables $(\bfS,\bfXi)$ using the generator of probability measures $P_\bfS$ (see Section~\ref{sec:Section5.5}) and $P_\bfXi$ (see Section~\ref{sec:Section5.4}-(i)). The spatial discretization of the weak formulation defined by Eq.~\eqref{eq:eq63} of the stochastic BVP and the discretization of Eq.~\eqref{eq:eq52} with Eq.~\eqref{eq:eq53} can be performed by the finite element method.
Using Section~\ref{sec:Section5.5}, $\lambda_1(\bfcurS^\kappa,\bfxi^\kappa) \geq \ldots \geq \lambda_6(\bfcurS^\kappa,\bfxi^\kappa) > 0$ are computed as the eigenvalues of $[\cc^\peff(\bfcurS^\kappa,\bfxi^\kappa)]\in\MM^+_6$. Taking into account Eq.~\eqref{eq:eq71}, the mean-square convergence of the random eigenvalues can be analyzed with the convergence function
\begin{equation}\label{eq:eq72}
\kappa_\psim\mapsto\pconv(\kappa_\psim) = \Vert\, [\underline\CC]\, \Vert_F^{-1} \,( \kappa_\psim^{-1} \sum_{\kappa=1}^{\kappa_\ppsim}\Vert \bflambda(\bfcurS^\kappa,\bfxi^\kappa)\Vert^2_2)^{1/2} \, .
\end{equation}
For a given tolerance of convergence, the probability density function (pdf) $\widetilde\bflambda\mapsto p_{\widetilde\bfLambda}(\widetilde\bflambda)$ on $\RR^6$ (with support $(\RR^{+*})^6$) with respect to the Lebesgue measure $d\widetilde\bflambda$ can be estimated  with the $\kappa_\psim$ independent realizations $\{\bflambda(\bfcurS^\kappa,\bfxi^\kappa), \kappa=1,\ldots ,\kappa_\psim\}$
using, for instance, the multidimensional Gaussian kernel-density estimation method \cite{Bowman1997}. The pdf $\widetilde\lambda_1\mapsto p_{\widetilde\Lambda_1}(\widetilde\lambda_1)$ of $\widetilde\Lambda_1$  can also be estimated yielding the pdf of the operator norm $\Vert\,[\widetilde\CC^\peff] \,\Vert_2  = \widetilde\Lambda_1$.
\section{Numerical illustration}
\label{sec:Section6}
\noindent {\textit{(i) Mean model of the microstructure}}. Domain $\Omega = (]0,1[)^3$ and the mean model of the elastic material is chosen in the orthotropic class with mean Young moduli $\underline E_1=10^{10}$, $\underline E_2=0.5\times 10^{10}$, and $\underline E_3=0.1\times 10^{10}$, with mean Poisson coefficients $\underline \nu_{23}=0.25$, $\underline \nu_{31}=0.15$, and $\underline \nu_{12}=0.1$ (the International System of Units is used).\\

\noindent {\textit{(ii) Elasticity random field}}.
The hyperparameter $\delta_c$ that allows for controlling the level of statistical fluctuations of $[\bfC(\bfx;\bfcurS)]$ (see Hypothesis~\ref{hypothesis:1}) is fixed to the value $0.4$.\\

\noindent {\textit{(iii) Uncertain spectral measure}}. The model of the spectral  measure is the one described in Example~\ref{example:1}. The probability measure $P_\bfW(d\bfw)$ of $\bfW$ (see Definition~\ref{definition:5}) is chosen as uniform on $\curC_w$. For all $j\in\{1,2,3\}$, the mean value of the random correlation length $L_{cj} = \pi/K_j$ is $\underline L_c$  and its coefficient of variation $\delta_{L_{cj}}$ is $\delta_{L_c}$, which are independent of $j$. Consequently, the support $[w_j^\pmin,w_j^\pmax]$ of the probability measure of $W_j = L_{cj}$ is such that $w_j^\pmin = \underline L_c\,(1-\sqrt{3} \,\delta_{L_{c}} ) $  and $w_j^\pmax = 2\,\underline L_c - w_j^\pmin$.
The hyperparameter $\delta_s$ that controls the level of uncertainties of the spectral measure, which is such that $\delta_s^2= (\Pi_{j=1}^3(1+\delta_j^2)) - 1$, is generated with $\delta_1=\delta_2=\delta_3=\delta_\unc$. A sensitivity analysis with respect to the level of spectrum uncertainties will be performed by considering $9$ values of the triplet of parameters $(\underline L_c, \delta_{L_c}, \delta_\unc)$ with $\underline L_c \in\{0.2,0.4,0.6\}$ and
$\delta_{L_c}= \delta_\unc\in\{0.2,0.3,0.4\}$. For this set of data, the minimum of correlation lengths is $0.06$ (obtained for $\underline L_c =0.2$ with
$\delta_{L_c}= \delta_\unc =0.4$) while the maximum is $1.01$ (obtained for $\underline L_c =0.6$ with $\delta_{L_c}= \delta_\unc =0.4$).
The spectral domain sampling for the discretization of the spectral measure is performed with $\nu_s=8$ and thus $\nu=8^3=512$. With the quadrant symmetry, we have $\widehat\nu_s = 4$ yielding $\widehat\nu=64$.\\

\noindent {\textit{(iv) Finite element discretization}}.
The weak formulation defined by Eq.~\eqref{eq:eq63} is discretized by the finite element method.
The finite element mesh is made up of $20\times 20\times 20 = 8\,000$ solid finite elements (8-nodes solid), $9\, 261$ nodes, and $27\, 783$ degrees of freedom (dof). There are $2\, 402$ nodes on the boundary and thus $7\, 206$ zeros Dirichlet conditions. There are $2^3$ integrations points in each finite element, which yields $64\, 000$ integrations points. The spatial discretization of the $\MM^+_6$-valued elasticity random field $[\bfC(\cdot,\bfcurS)]$ yields
$21\times 64\, 000= 1\, 344\, 000$ random terms (taking into account the symmetry).\\

\noindent {\textit{(v) Stochastic solver}}. The Monte Carlo simulation method is performed for $\kappa_\psim \in [1, 2\,000]$.
Figure~\ref{fig:figure1}-(left) displays the convergence function $\kappa_\psim\mapsto\pconv(\kappa_\psim)$ defined by Eq.~\eqref{eq:eq72}  for $\underline L_c =0.2$ and with no uncertainty in the spectral measure ($\delta_{L_c}= \delta_\unc = 0$) and  for $\underline L_c \in\{0.2,0.4,0.6\}$ with the largest uncertainties in the spectral measure, $\delta_{L_c}= \delta_\unc = 0.4$, which is the most unfavorable value with respect to convergence.  It can be seen that the mean-square convergence is obtained for $\kappa_\psim = 2\,000$.\\

\begin{figure}[tb]
\includegraphics[width=6.5cm]{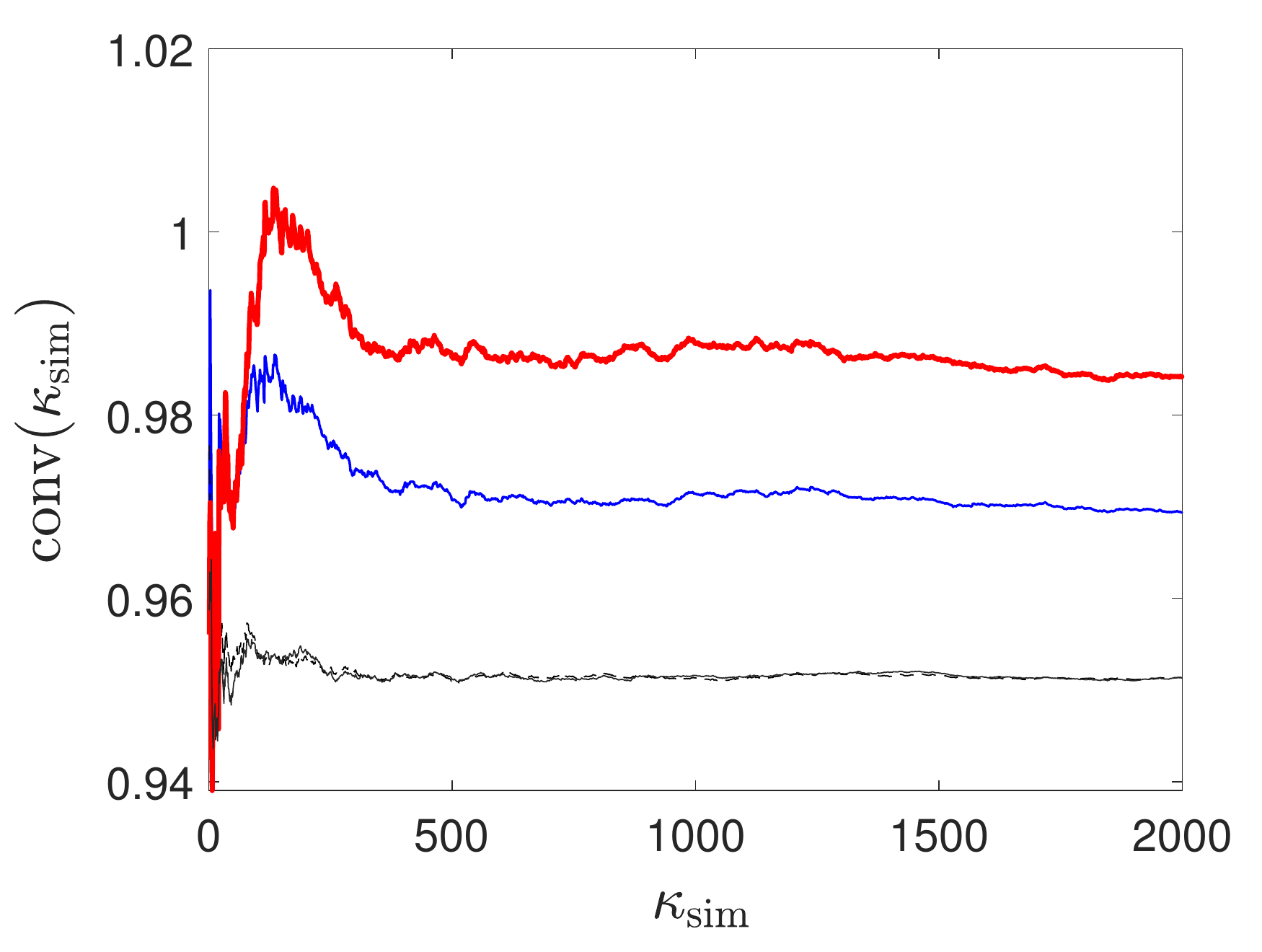} \includegraphics[width=6.7cm]{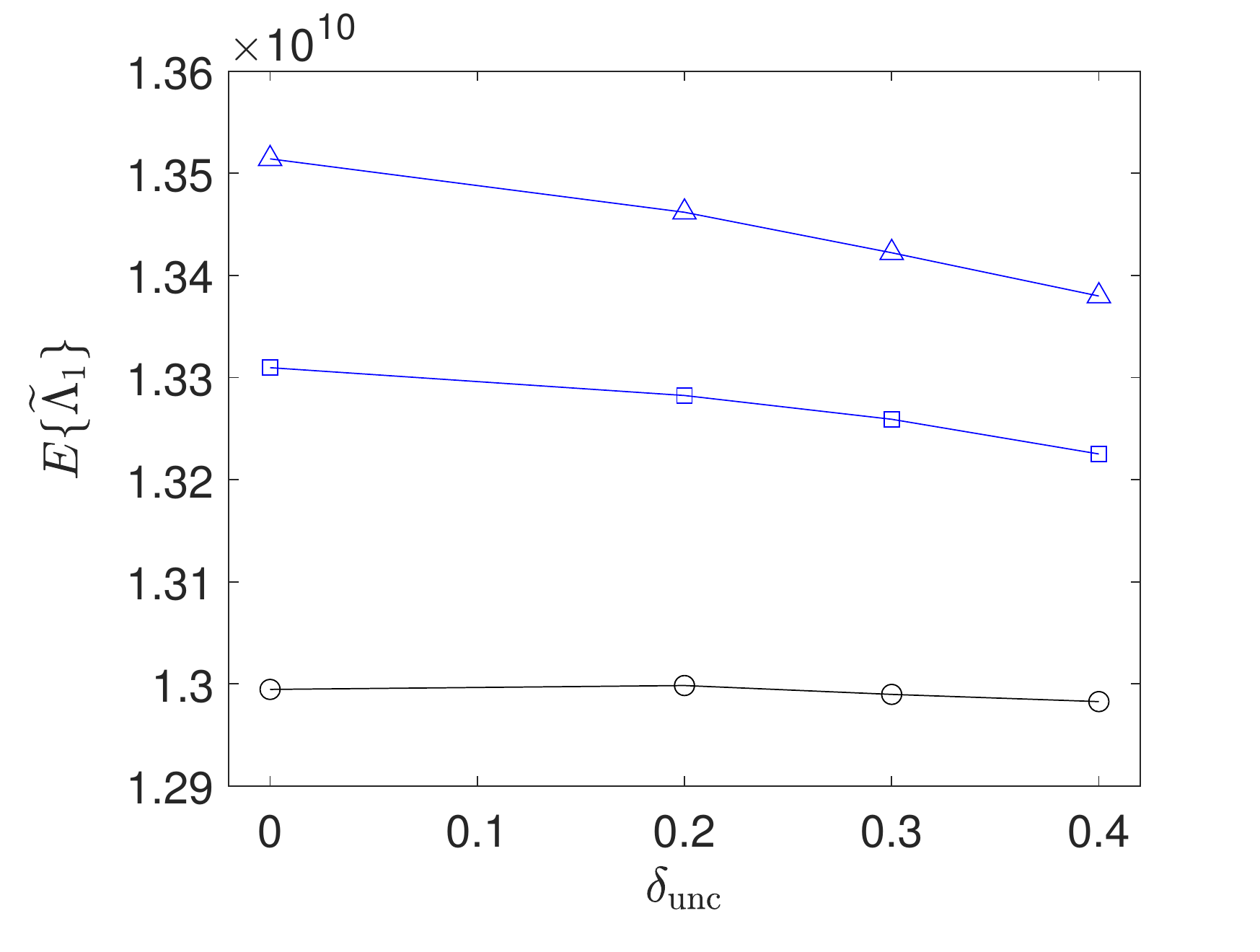}
\caption{Left figure: function $\kappa_\psim\mapsto\pconv(\kappa_\psim)$ for
$\delta_\unc = 0$ and $\underline L_c =0.2$ (dashed line);
$\delta_\unc = 0.4$ and $\underline L_c =0.2$ (black thin line), $0.4$ (blue med line), and $0.6$ (red thick line).
Right figure: $E\{\widetilde\Lambda_1\}$ as a function of $\delta_\unc$ for $\underline L_c = 0.2$ (circle marker),
$0.4$ (square marker), and $0.6$ (triangle-up marker).}
\label{fig:figure1}
\end{figure}
\begin{figure}[tb]
\includegraphics[width=6.5cm]{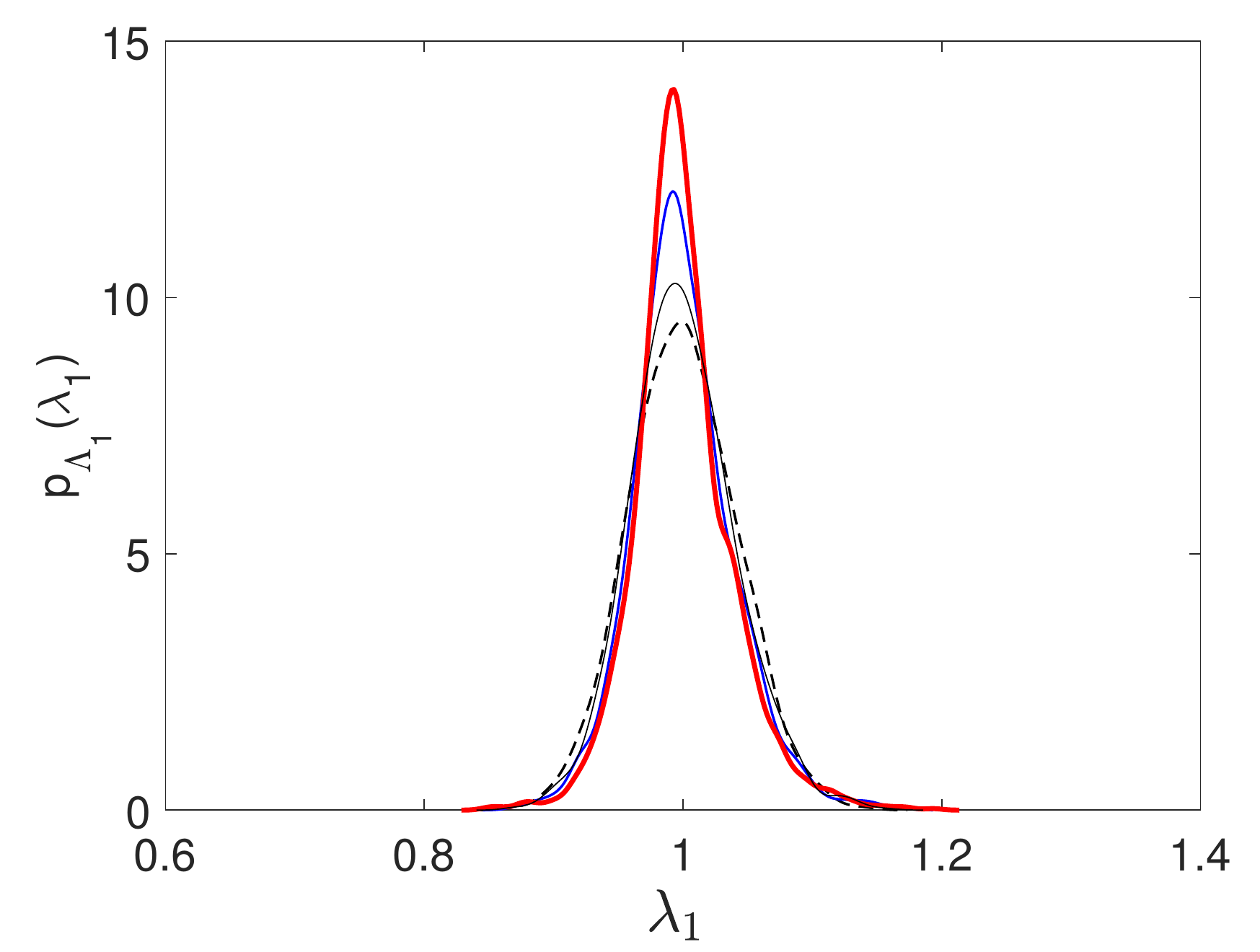}  \includegraphics[width=6.5cm]{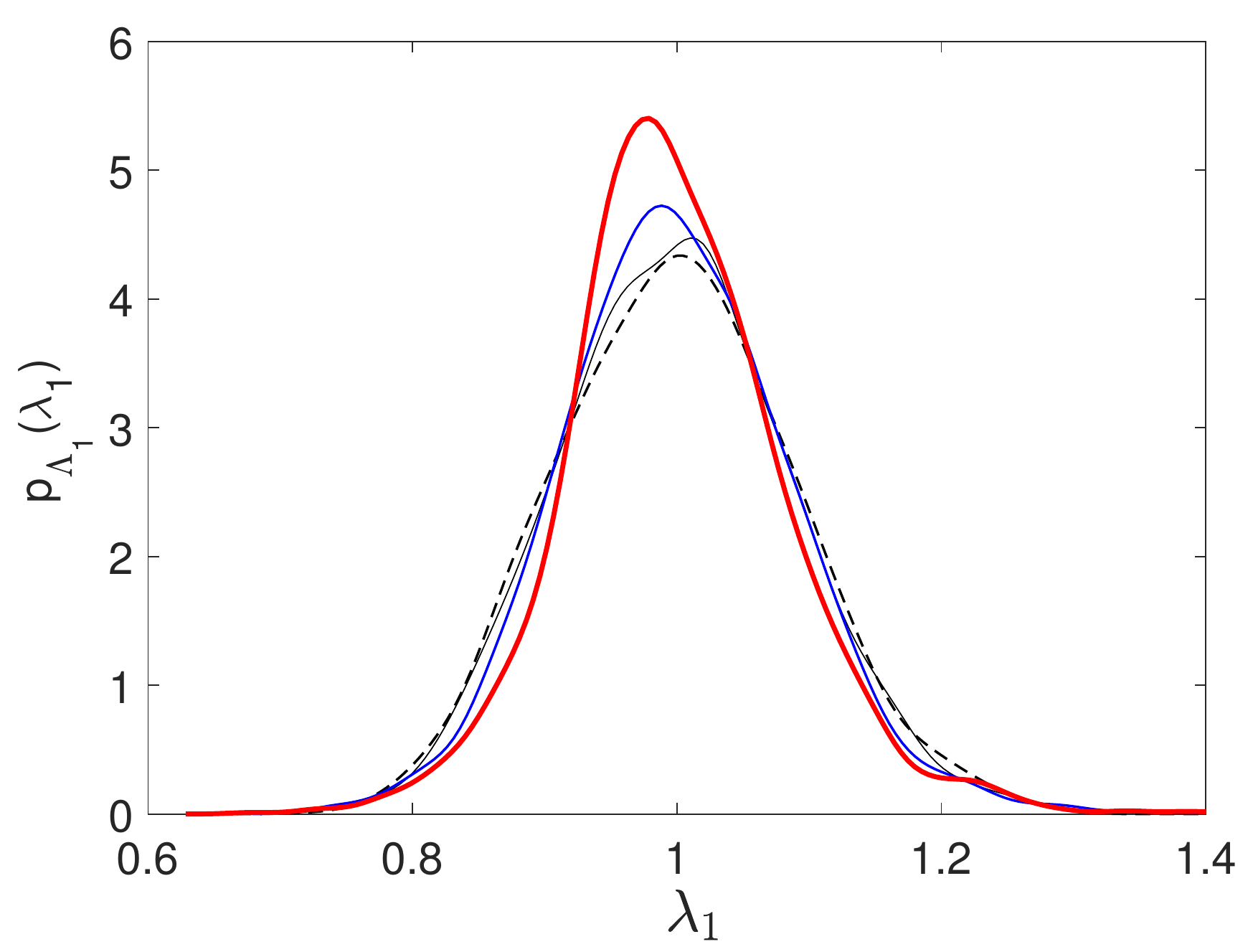} \includegraphics[width=6.5cm]{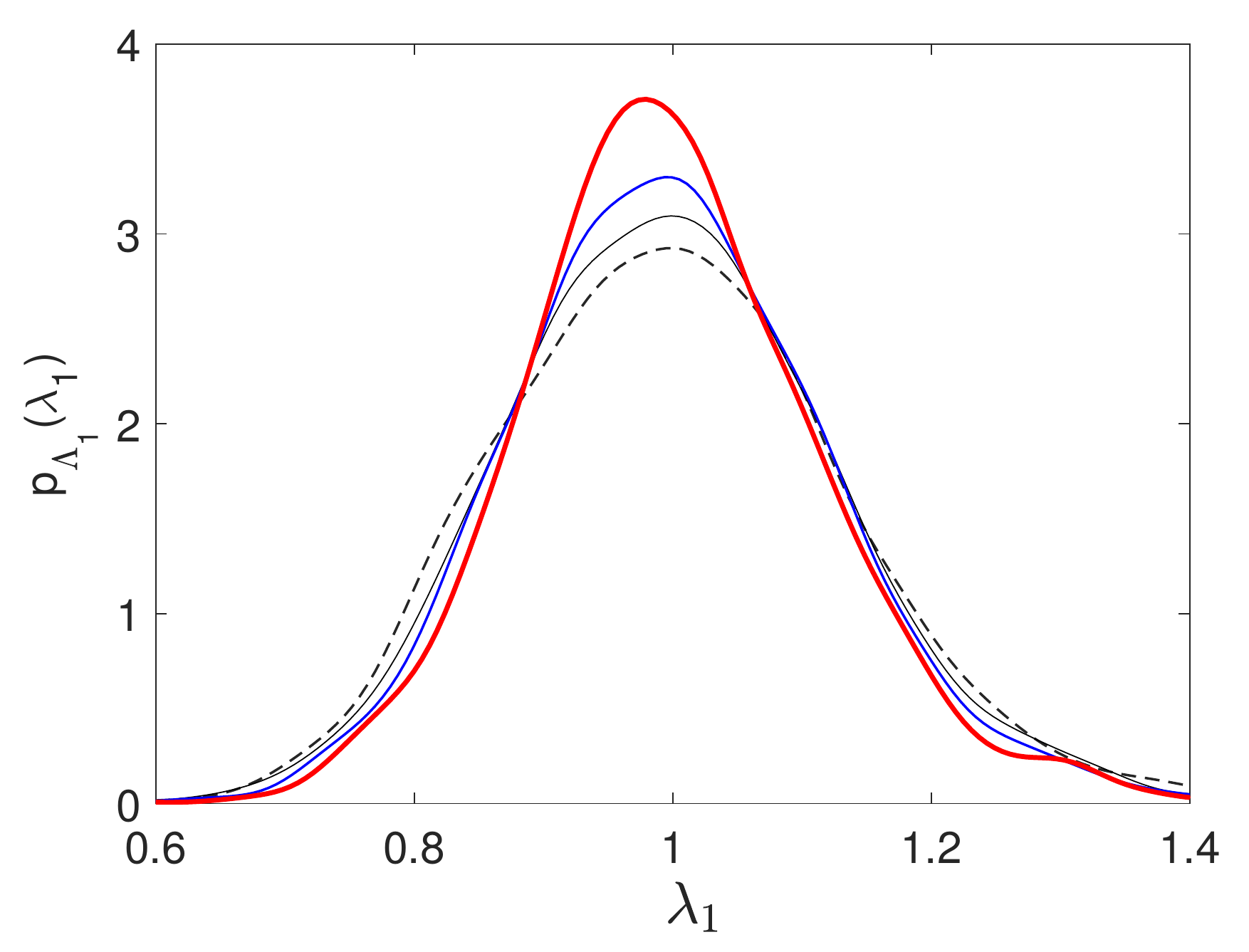}
\caption{Sensitivity of the pdf $\lambda_1 \mapsto p_{\Lambda_1}(\lambda_1)$
for $\underline L_c =0.2$ (top left figure), $0.4$ (top right figure), and $0.6$ (down figure) as a function of the spectral measure  uncertainty level:
$\delta_\unc = 0$ (dashed line), $0.2$ (black thin line), $0.3$ (blue med line), $0.4$ (red thick line).}
\label{fig:figure2}
\end{figure}

\noindent {\textit{(vi) Sensitivity of the probability density function of the normalized operator norm of the random effective elasticity matrix as a function of the uncertainty level of the spectral measure}}.
Let $\Lambda_1 = \widetilde\Lambda_1/E\{\widetilde\Lambda_1\}$ be the normalized operator norm $\Vert\, [\widetilde\CC^\peff]\,\Vert / E\{[\widetilde\CC^\peff]\,\Vert\} =\Lambda_1$. For $\underline L_c =0.2$, $0.3$, $0.4$, Fig.~\ref{fig:figure1}-(right) displays the graph of $E\{\widetilde\Lambda_1\}$ as a function of the level $\delta_{L_c}= \delta_\unc$ of the spectral measure uncertainties.
Figure~\ref{fig:figure2} shows the pdf $\lambda_1\mapsto p_{\Lambda_1}(\lambda_1)$ of the normalized operator norm of the random effective elasticity matrix
for $\underline L_c =0.2$, $0.4$, and $0.6$, with no uncertainty  in the spectral measure ($\delta_{L_c}= \delta_\unc = 0$) and with uncertainties $\delta_{L_c}= \delta_\unc = 0.2$, $0.3$, and $0.4$.\\

\noindent {\textit{(vii) Sensitivity of the probabilistic analysis of the RVE size with respect to the uncertainty level of the spectral measure}}.
We analyze the random largest eigenvalue $\Lambda_1$  (normalized operator norm).
Let $\eta$ be a positive real number and let $\eta\mapsto P(\eta)$ be the function from $]0,1]$ into $[0,1]$ defined by
\begin{equation}\label{eq:eq73}
P(\eta) =\hbox{Proba}\{1-\eta < \Lambda_1 \leq 1 +\eta\} = F_{\Lambda_1}(1+\eta) - F_{\Lambda_1}(1-\eta)\, ,
\end{equation}
in which $F_{\Lambda_1}$ is the cumulative distribution function of $\Lambda_1$.
Figure~\ref{fig:figure3}  shows the sensitivity of the graph of function $\eta\mapsto P(\eta)$ for
$\underline L_c = 0.2$, $0.4$, and $0.6$ with respect to the level of uncertainties in the spectral measure for
$\delta_{L_c}= \delta_\unc = 0$, $0.2$, $0.3$, and $0.4$.
Table~\ref{table1} yields an extraction from Fig.~\ref{fig:figure3} of the probability levels.\\

\begin{figure}[tb]
\includegraphics[width=6.5cm]{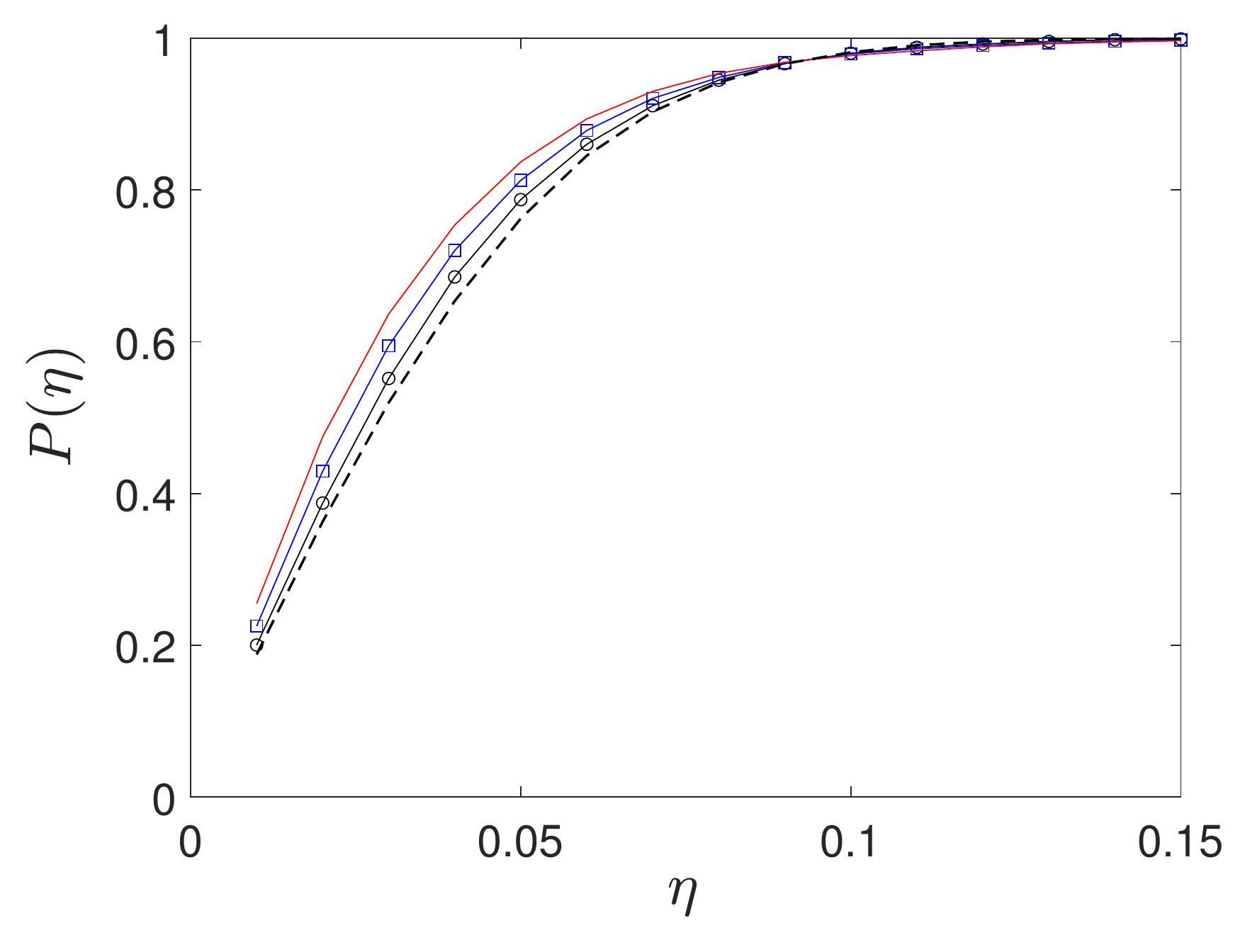}  \includegraphics[width=6.5cm]{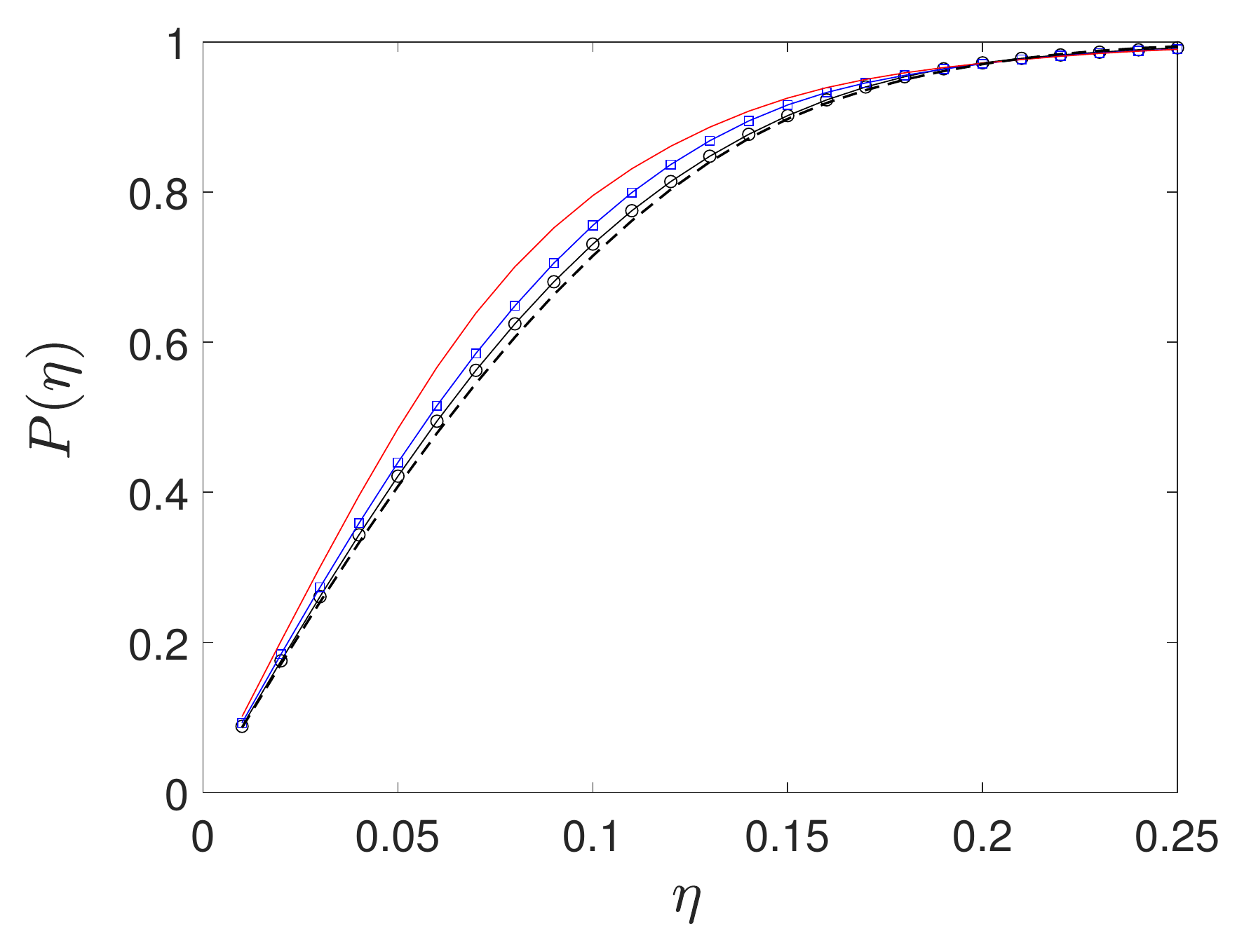} \includegraphics[width=6.5cm]{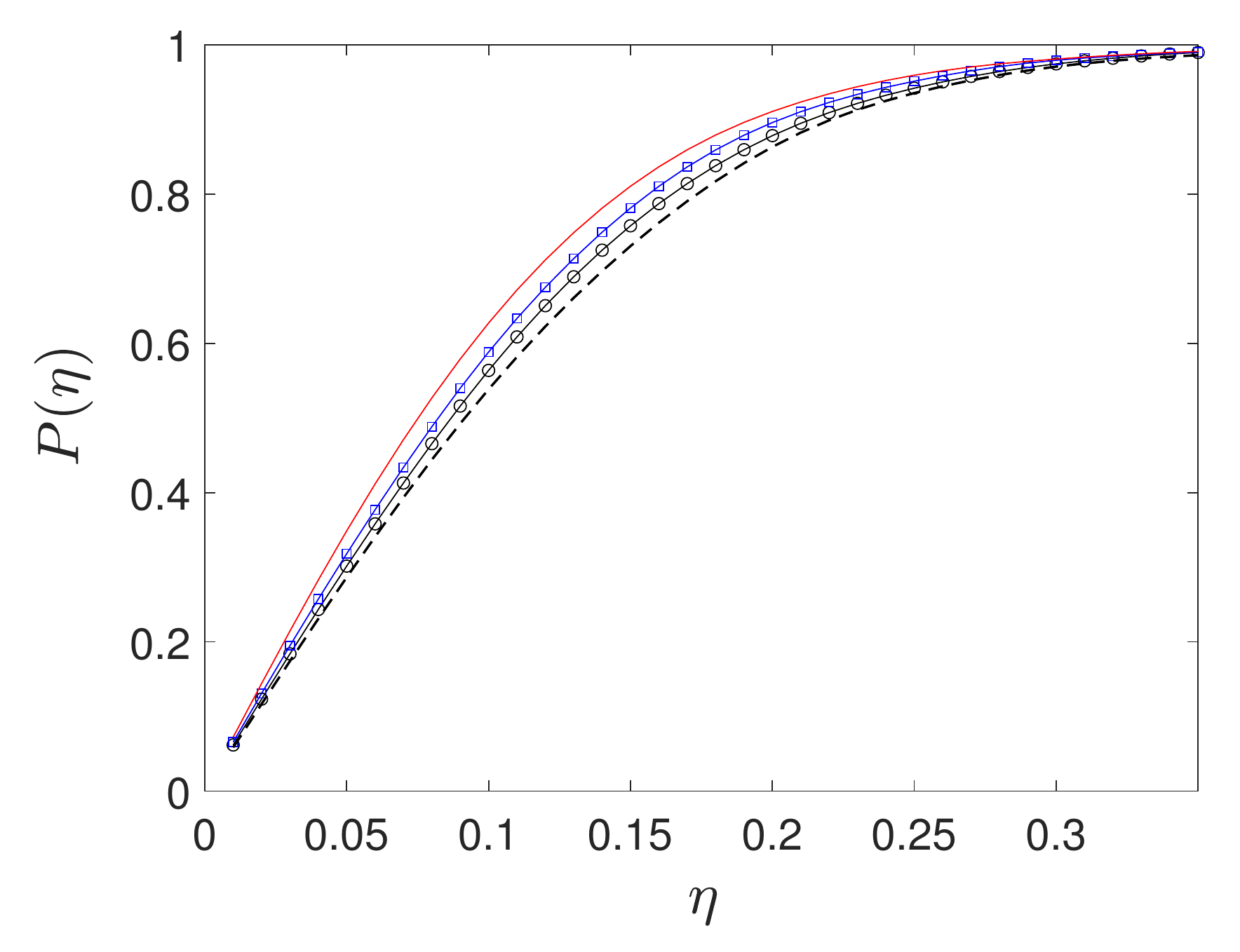}
\caption{Sensitivity of the graph of function $\eta\mapsto P(\eta)$ for
$\underline L_c = 0.2$ (top left figure), $0.4$ (top right figure), and $0.6$ (down figure)
with respect to the level of uncertainties in the spectral measure for
$\delta_\unc = 0$ (dashed line), $0.2$ (circle marker), $0.3$ (square marker), and $0.4$ (no marker).}
\label{fig:figure3}
\end{figure}
\begin{table}[ht]
  \caption{Sensitivity of the probabilistic analysis of the RVE size with respect to the uncertainty level of the spectral measure.}\label{table1}
\begin{center}
 \begin{tabular}{|c|c|c|c|c|} \hline
 $\underline L_c$ & $\delta_\unc$ & $P\{0.98 < \Lambda_1 \leq 1.02\}$ & $P\{0.96 < \Lambda_1 \leq 1.04\}$   & $P\{0.92 < \Lambda_1 \leq 1.08\}$       \\
 \hline
        0.2      &       0.0     &      0.365                         &    0.655                                &   0.942                             \\
                 &       0.2     &      0.385                         &    0.683                                &   0.945                             \\
                 &       0.3     &      0.430                         &    0.721                                &   0.948                             \\
                 &       0.4     &      0.475                         &    0.705                                &   0.954                             \\
 \hline
        0.4      &       0.0     &      0.171                         &    0.332                                &   0.610                             \\
                 &       0.2     &      0.175                         &    0.344                                &   0.625                             \\
                 &       0.3     &      0.185                         &    0.360                                &   0.650                             \\
                 &       0.4     &      0.202                         &    0.395                                &   0.700                             \\
 \hline
        0.6      &       0.0     &      0.108                         &    0.230                                &   0.442                             \\
                 &       0.2     &      0.125                         &    0.240                                &   0.465                             \\
                 &       0.3     &      0.130                         &    0.260                                &   0.488                             \\
                 &       0.4     &      0.145                         &    0.280                                &   0.526                             \\
 \hline
  \end{tabular}
\end{center}
\end{table}

\noindent {\textit{(viii) Brief discussion}}. When there are no uncertainties in the spectral measure ($\delta_\unc = 0$), Figure~\ref{fig:figure3} and Table~\ref{table1} shows that the condition to obtain a scale separation is not really obtained because, for $\underline L_c = 0.2$ and $\delta_\unc = 0$ it can be seen that
$\hbox{Proba}\{0.98 < \Lambda_1 \leq 1.02\} = 0.365$ and the probability becomes $0.942$ only for $\{0.92 < \Lambda_1 \leq 1.08\}$.
The results show in Table~\ref{table1} shows that, for the specific case analyzed (in particular choosing the same level of uncertainties for the spatial correlation lengths and for the values of the spectral density function) and contrary to what was expected, the introduction of spectral measure uncertainties improves the scale separation from a probabilistic analysis point of view.
In fact, for $\underline L_c= 0.2$ and for $\delta_{L_c}= \delta_\unc =0.4$, the minimum value of the realizations of the random correlation length is $0.06$, value less than $10$ percent of the characteristic length of the dimensions of domain $\Omega$, for which the scale separation can be obtained.
It should be noted that these results are presented as an illustration of the use of the proposed mathematical construction of a random field  with uncertainties in the spectral measure in order to improve the probabilistic analysis of stochastic homogenization of random elastic media.
More advanced computational analyses should be performed with this probabilistic model in order to deeply analyze the role played by uncertain spectral measure for stochastic homogenization of random elastic media, in particular in taking different values of uncertainties for the spatial correlation lengths and for the spectral density function.
%
% Bibliography

%  For BibTeX users:
% \bibliographystyle{elsarticle-num}
% \bibliography{reference}

\end{document}